\documentclass[10pt]{mystyle}

\title[Localization in Statistical Learning]{Localization of VC Classes: Beyond Local Rademacher Complexities}
\usepackage{times}
\usepackage{bbm}
\usepackage{mathtools}
\usepackage{comment}

 \newcommand{\E}{\mathbb{E}} 
 \renewcommand{\P}{\mathbb{P}}
  
 \newcommand{\Ind}{\mathbbm{1}} 
 \newcommand{\F}{\mathcal{F}}
  \newcommand{\G}{\mathcal{G}}
 \newcommand{\sign}{\text{sign}}
  \newcommand{\dis}{\text{DIS}}
    \newcommand{\mstar}{\text{star}}
\newcommand{\mycomment}[1]{}
\newcommand{\Prob}{\mathsf{P}}
\bibpunct{[}{]}{,}{n}{,}{,}

 \author{\Name{Nikita Zhivotovskiy} \Email{nikita.zhivotovskiy@phystech.edu}\\
 \addr Moscow Institute of Physics and Technology and Institute for Information Transmission Problems, Moscow, Russia
 \AND
 \Name{Steve Hanneke} \Email{steve.hanneke@gmail.com} 
 }

\begin{document}

\maketitle

\begin{abstract}
In statistical learning the excess risk of empirical risk minimization (ERM) is controlled by $\left(\frac{\text{COMP}_{n}(\F)}{n}\right)^{\alpha}$, where $n$ is a size of a learning sample, $\text{COMP}_{n}(\F)$ is a complexity term associated with a given class $\F$ and $\alpha \in [\frac{1}{2}, 1]$ interpolates between slow and fast learning rates. In this paper we introduce an alternative localization approach for binary classification that leads to a novel complexity measure: fixed points of the local empirical entropy. We show that this complexity measure gives a tight control over $\text{COMP}_{n}(\F)$ in the upper bounds under bounded noise. Our results are accompanied by a minimax lower bound that involves the same quantity. In particular, we practically answer the question of optimality of ERM under bounded noise for general VC classes. \end{abstract}


\section{Introduction}
Since the early days of statistical learning theory understanding of the generalization abilities of empirical risk minimization has been a central question.  In 1968, Vapnik and Chervonenkis \cite{vapnik68} introduced the combinatorial property of classes of classifiers which we now call the \emph{VC dimension}, which plays a crucial role not only in statistics but in many other areas of mathematics.  By now it is strongly believed that the VC-dimension fully characterizes the properties of the empirical risk minimization algorithm. For example, when no restrictions are made on the distributions one can prove that the probability of error of the minimizer of empirical risk is close to the probability of error of the best classifier in the class, up to a term of order $\sqrt{\frac{d}{n}} + \sqrt{\frac{\log(\frac{1}{\delta})}{n}}$, with probability at least $1 - \delta$, where $d$ is the VC dimension of the class and $n$ is the sample size. One can also prove a minimax lower bound (valid for any learning procedure) matching up to absolute constants. But the fact that VC dimension alone describes the complexity term appears to be true only in the agnostic case, when no assumptions are made on the labelling mechanism.  It was noticed several times in the literature, that when considering bounded noise, VC dimension alone is not a right complexity measure of ERM \cite{Massart06,Raginsky11, Hanneke15a}. Until now an exact right complexity measure has only been identified for a few specific classes. In this paper we propose a complexity measure which
provides upper bounds on the risk of ERM, as well as lower bounds under regularity conditions,
and therefore represents the right complexity measure for ERM in these cases. 

In the last twenty years many efforts were made to understand the conditions that imply fast $\frac{1}{n}$ convergence rates, instead of slow $\frac{1}{\sqrt{n}}$ rates. By now these conditions are well understood; we refer for example to van Erven et al. \cite{VanErven15} for an extensive survey and related results. At the beginning of the 2000s, so-called \emph{localized} complexities (Bartlett et al. \cite{Bartlett05}, Koltchinskii \cite{Koltch06}) were introduced to statistical learning and became popular techniques for proving $\frac{1}{n}$ rates in different scenarios. But in addition to better rates, localization means that \emph{only a small vicinity of the best classifier} really affects the learning complexity. Almost fifty years after the introduction of VC theory this phenomenon is still not fully understood and studied. Specifically, we lack tight error bounds based on localization and expressed in terms of intuitively-simple and calculable combinatorial properties of the class.  Existing approaches based on localization (mainly, via \emph{local Rademacher complexities}) are typically difficult to calculate directly, and the simpler relaxations of these bounds in the literature use localization merely to gain improvements due to the \emph{noise conditions}, but fail to maintain the important improvements due to the local \emph{structure of the function class} (i.e., localization of the complexity term in the bound). Moreover, to the best of our knowledge, in classification literature there are no known general minimax lower bounds in terms of localized processes.

There does exist one line of results which simultaneously give fast convergence rates and perform direct localization of a class of classifiers, to arrive at simple generalization bounds. Specifically, Massart and N\'ed\'elec \cite{Massart06} proved that under Massart's bounded noise condition, generalization of order $\frac{d}{nh}\log(\frac{nh^2}{d}) + \frac{\log(\frac{1}{\delta})}{nh}$ is possible, where $h$ is a margin parameter responsible for the noise level. To derive this bound, Massart and N\'{e}d\'{e}lec use a localized analysis to obtain improved rates under these noise conditions.  However, the bound does not reflect this localization in the \emph{complexity term} itself: in this case, the factor $d\log(\frac{nh^2}{d})$. Gin\'e and Koltchinskii \cite{Gine06} refined this bound, establishing generalization of order $\frac{d}{nh}\log(\tau(\frac{d}{nh^2})) + \frac{\log(\frac{1}{\delta})}{nh}$ for empirical risk minimization, where $\tau$ is a distribution-dependent quantity they refer to as \emph{Alexander's capacity function} (from the work of Alexander in the 80s \cite{Alexander87}). Very recently, Hanneke and Yang \cite{Hanneke15} introduced a novel combinatorial parameter $\mathbf{s}$, called the \emph{star number}, which gives perfectly-tight distribution-free control on $\tau(\frac{d}{nh^2})$, and generally cannot be upper bounded in terms of the VC dimension.  Thus (as noted by Hanneke \cite{Hanneke15a}), in terms of distribution-free guarantees on the generalization of empirical risk minimization, the implication of Gin\'e and Koltchinskii's result is a bound $\frac{d}{nh}\log(\mathbf{s} \land \frac{nh^2}{d}) + \frac{\log(\frac{1}{\delta})}{nh}$.  However, this bound is sometimes suboptimal. In this paper we will give a new argument showing potential gaps of this bound.

The aim of this paper is to perform a tight distribution-free localization for VC classes under bounded noise by introducing an appropraite distribution-free complexity measure, thus resolving the existing gap between upper and lower bounds.  The complexity measure is a localized empirical entropy measure: essentially, a fixed point of the local empirical entropy. Most of the results will be proved in expectation and in deviation. Although results in expectation can usually be derived by integrating the results in deviation, we will directly prove results in expectation in the main part of the paper. Proofs of standard technical propositions and some results in deviation will be moved to the appendix. This paper is organized as follows:
\begin{itemize}
\item In section $2$ we introduce the notation, definitions and previous results.
\item In section $3$ we introduce and further develop the machinery, based on the combination of shifted empirical processes \cite{Lecue12} and offset Rademacher complexities \cite{Liang15}.  We also obtain a new upper bound on the error rate of empirical risk minimization in the realizable case, involving the star number and the growth function, which refines a recent result of Hanneke \cite{Hanneke15a} in some cases; this bound is a strict improvement over the distribution-free bound implied by the result of Gin\'{e} and Koltchinskii in the realizable case.

\item Section $4$ is devoted to an upper bound in terms of fixed point of global metric entropy. Although it gives a fast convergence rate $\frac{1}{n}$, it involves only a global information about the class. Thus, this bound is suboptimal in some interesting cases, as are the other bounds in the literature based solely on global complexities for the class.  We include the proof nevertheless, as it cleanly illustrates certain aspects of our approach; for simplicity, we only present this result in the realizable case.
 \item Section $5$ contains our main results. In this section we introduce the local empirical entropy and prove that fixed points of local empirical entropy control the complexity of ERM under bounded noise. 
\item Section $6$ is devoted to a novel lower bound in terms of fixed points of local empirical entropy under mild regularity assumptions.
\item Section $7$ contains examples of values of fixed points for some standard classes.
\item Section $8$ is devoted to discussions and some related general results. Specifically, we prove that bounds based on our complexity measure are always not worse than the bounds based on local Rademacher complexities.
\end{itemize}

\section{Notation and Previous Results}
We define the \emph{instance space} $\mathcal{X}$ and the \emph{label space} $\mathcal{Y} = \{1, -1\}$. We assume that the set $\mathcal{X} \times \mathcal{Y}$ is equipped with some $\sigma$-algebra and a probability measure $P$ on measurable subsets is defined. We also assume that we are given a set of classifiers $\F$; these are measurable functions with respect to the introduced $\sigma$-algebra, mapping $\mathcal{X}$ to $\mathcal{Y}$. We may always decompose $P = P_{X}\times P_{Y|X}$. The risk of a classifier $f$ is its probability of error, denoted $R(f) = P(f(X) \neq Y)$. It is known that among all functions the \emph{Bayes classifier} $f^{*}(x) = \sign(\eta(x))$, where $\eta(x) = \E[Y|X = x]$, minimizes the risk \cite{Devr95}. Symbol $\wedge$ will denote minimum of two real numbers, $\vee$ will denote maximum of two real numbers and $\Ind[A]$ will denote an indicator of the event $A$. For any subset $B \subseteq \F$ define the \emph{region of disagreement} as $\dis(B) = \{x \in \mathcal{X}|\ \exists f, g \in B\ \text{s. t.}\ f(x) \neq g(x)\}$. We will also consider abstract real-valued functional classes, which will usually be denoted by $\G$. We will slightly abuse the notation and by $\log(x)$ always mean truncated logarithm: $\ln(\max(x, e))$. The notation $f(n) \lesssim g(n)$ or $g(n) \gtrsim f(n)$ will mean that for some universal constant $c>0$ it holds that $f(n) \le cg(n)$ for all $n \in \mathbb{N}$. Similarly,  we introduce $f(n) \simeq g(n)$ to be equivalent to $g(n) \lesssim f(n) \lesssim g(n)$. 

A \emph{learner} observes $\left((X_{1}, Y_{1}), \ldots, (X_{n}, Y_{n})\right)$, an i.i.d. training sample from an unknown distribution $P$. Also denote $Z_i = (X_i,Y_i)$ and $\mathcal{Z} = \mathcal{X} \times \mathcal{Y}$. By $P_{n}$ we will denote expectation with respect to the empirical measure (empirical mean) induced by these samples. \emph{Empirical risk minimization} (ERM) refers to any learning algorithm with the following property: given a training sample, it outputs a classifier $\hat{f}$ that minimizes $R_{n}(f) = P_{n}\Ind[f(X) \neq Y]$ among all $f \in \F$. Depending on context we will usually refer to $\hat{f}$ as an empirical risk minimizer and use the same abbreviation.
At times we also refer to a \emph{ghost sample}, which is another $n$ i.i.d. $P$-distributed samples, 
independent of the training sample, and we denote by $P'_{n}$ the empirical mean with respect to the ghost sample.
We say a set $\{x_1, \ldots, x_k\} \in \mathcal{X}^{k}$ is shattered by $\F$ if there are $2^k$ distinct classifications of $\{x_1, \ldots, x_k\}$ realized by classifiers in $\F$. The \emph{VC dimension} of $\F$ is the largest integer $d$ such that there exists a set $\{x_1, \ldots, x_d\}$ shattered by $\F$ \cite{vapnik68}. We define the \emph{growth function} $\mathcal{S}_{\mathcal{F}}(n)$ as the maximum possible number of different classifications of a set of $n$ points realized by classifiers in $\F$ (maximized over the choice of the $n$ points).
Throughout the paper $n$ will always denote the size of the training sample, $d$ will denote the VC dimension, and $\hat{f}$ will denote the output of any ERM algorithm.  To focus on nontrivial scenarios, we will always suppose $d \geq 1$.
In what follows we adopt the assumption that the events appearing in probability claims below are measurable.
\begin{definition}[Massart and N\'ed\'elec \cite{Massart06}]
$(P, \F)$ is said to satisfy Massart's bounded noise condition if $f^{*} \in \F$ and for some $h \in [0, 1]$ it holds $|\eta(X)| \ge h$ with probability $1$. This constant $h$ is referred to as the \emph{margin parameter}.
\end{definition}
For any $\F$, the set of all corresponding distributions satisfying Massart's bounded noise condition will be denoted by $\mathcal{P}(h, \F)$. The case $h = 1$ corresponds to the so-called \emph{realizable case}, where $Y = f^{*}(X)$ almost surely, and $h = 0$ corresponds to a well-specified (i.e., $f^* \in \F$) noisy case. The following result is classic \cite{Dudley84, Talagrand94,Boucheron05}. Let $\F$ be a class with VC-dimension $d$. For any empirical risk minimizer $\hat{f}$ over $n$ samples, for any $P \in \mathcal{P}(0,\F)$, with probability at least $1 - \delta$,
\[
R(\hat{f}) - R(f^{*}) \lesssim \sqrt{\frac{d}{n}} + \sqrt{\frac{\log(\frac{1}{\delta})}{n}}.
\]
Moreover, the following lower bound exists for an output $\tilde{f}$ of \emph{any} algorithm based on $n$ samples: there exists $P \in \mathcal{P}(0,\F)$ such that, with probability greater than $1 - \delta$,
\[
R(\tilde{f}) - R(f^{*}) \gtrsim \left(\sqrt{\frac{d}{n}} + \sqrt{\frac{\log(\frac{1}{\delta})}{n}}\right) \wedge 1.
\]
Thus we know that the VC-dimension is the right complexity measure for empirical risk minimization, and indeed for optimal learning, when no restrictions are made on the  probability distribution. Interestingly, this is not generally the case when $h > 0$. In this paper, we find this yet unknown essentially correct complexity measure, when $h$ is bounded away from $0$ and $1$.
But first, we review a refinement to the above bound for the case $h > 0$, due to Gin\'e and Koltchinskii \cite{Gine06}.  Specifically, consider the following definition.
\begin{definition}
\label{def:alexander}
For $\varepsilon_{0} > 0$ fix a set $\F_{\varepsilon_{0}} = \bigl\{f \in \F: P_{X}\bigl(f(X) \neq f^{*}(X)\bigr)\leq \varepsilon_0\bigr\}$. 
For $\varepsilon \in (0,1]$ define 
\[
\tau(\varepsilon) = 
\sup\limits_{\varepsilon_{0} \ge \varepsilon}
\frac{P_{X}\{x \in \mathcal{X} \colon \exists f\in \F_{\varepsilon_{0}} \text{ s.t. } f(x) \neq f^{*}(x)\}}
{\varepsilon_{0}} \lor 1.
\]
\end{definition}
This quantity (essentially\footnote{The original definition did not include the supremum over $\varepsilon_{0}$, instead taking $\varepsilon_{0}=\varepsilon$ directly.  However, the results were proven under a very restrictive monotonicity assumption. Taking the supremum allows one to dispense with such assumptions.}) was introduced to the empirical processes literature by Alexander \cite{Alexander87}, and is referred to as \emph{Alexander's capacity} by Gin\'e and Koltchinskii \cite{Gine06}.  The same quantity appeared independently in the literature on active learning, where it is referred to as the \emph{disagreement coefficient} \cite{Hanneke07,Hanneke14}.  $\tau(\varepsilon)$ is a distribution-dependent measure of the diversity of ways in which classifiers in a relatively small vicinity of $f^{*}$ can disagree with $f^{*}$. Gin\'e and Koltchinskii \cite{Gine06} gave the following upper bound. Let $\F$ be a class of VC dimension $d$, and $\hat{f}$ the classifier produced by an ERM based on $n$ training samples. For any probability measure $P \in \mathcal{P}(h, \F)$, with probability at least $1-\delta$,
\begin{equation}
\label{koltchbound}
R(\hat{f}) - R(f^{*}) \lesssim \frac{d}{nh}\log\left(\tau\left(\frac{d}{nh^2}\right)\right) + \frac{\log(\frac{1}{\delta})}{nh}.
\end{equation}
This bound is the best simple, easily calculable upper bound known so far for ERM in the case of binary classification under Massart's bounded noise condition. The proof of this bound is based on the analysis of the localized Rademacher processes. So we may also consider this result as the best relaxation of the local Rademacher analysis.

Recently, Hanneke and Yang \cite{Hanneke15} introduced a distribution-free complexity measure, called the \emph{star number}, which perfectly captures the worst case value for Alexander's capacity.  It is defined as follows.
\begin{definition}
The star number $\mathbf{s}$ is the largest integer such that there exist distinct $x_{1}, \ldots, x_{\mathbf{s}}\! \in\! \mathcal{X}$ and $f_{0}, f_{1}, \ldots, f_{\mathbf{s}} \in \F$ such that, for all $i\! \in\! \{1, \ldots, \mathbf{s}\}$, $\dis(\{f_{0}, f_{i}\}) \cap \{x_{1}, \ldots, x_{\mathbf{s}}\} = \{x_i\}$.
\end{definition}
Similar to Alexander's capacity, the star number describes how diverse the small-size disagreements with a fixed classifier $f_{0}$ can be. In terms of the one-inclusion graph studied by Haussler, Littlestone, and Warmuth \cite{Haussler94}, the star number may be described as the maximum possible degree in the
data-induced one-inclusion graph. It is easy to see that, for any class of VC dimension $d$, it always holds that $d \le \mathbf{s}$, but the difference may be as large as inifinte. We refer to \cite{Hanneke15} for examples and further discussions related to the star number. One of the most interesting results about this value is its connection with the worst case of Alexander's capacity. The paper of Hanneke and Yang contains the following equality
\begin{equation}
\sup\limits_{f^{*} \in \F} \sup\limits_{P_{X}}\tau(\varepsilon) = \mathbf{s} \wedge \frac{1}{\varepsilon}.
\end{equation}
As noted by Hanneke \cite{Hanneke15a}, an immediate corollary of this and \eqref{koltchbound} is that, for any $P \in \mathcal{P}(h,\F)$, with probability at least $1-\delta$,
\begin{equation}
\label{Ginecorollary}
R(\hat{f}) - R(f^{*}) \lesssim \frac{d}{nh}\log\left(\frac{nh^2}{d} \wedge \mathbf{s}\right) + \frac{\log(\frac{1}{\delta})}{nh}.
\end{equation}
In particular, in the realizable case (when $h = 1$), with probability at least $1-\delta$, 
\begin{equation*}
R(\hat{f}) \lesssim \frac{d}{n}\log\left(\frac{n}{d} \wedge \mathbf{s}\right) + \frac{\log(\frac{1}{\delta})}{n}.
\end{equation*}
Since $\mathbf{s}$ controls Alexander's capacity with equality, there is no room for any kind of improvement using the bound of Gin\'e and Koltchinskii if we consider distribution-free upper bounds. However, the above bound for the realizable case has recently been refined by Hanneke \cite{Hanneke15a}, establishing that for any $P \in \mathcal{P}(1,\F)$, with probability at least $1-\delta$, 
\begin{equation}
\label{s-over-d-bound}
R(\hat{f}) \lesssim \frac{d}{n}\log\left(\frac{n}{d} \wedge \frac{\mathbf{s}}{d}\right) + \frac{\log(\frac{1}{\delta})}{n}.
\end{equation}
Even this slight improvement indicates the suboptimality of the bound \eqref{Ginecorollary}. In this paper we will further refine this bound and discuss in details the following fact: the pair $d, \mathbf{s}$ alone is not a right complexity measure for the VC classes when $h$ is bounded away from zero.

\section{Preliminaries from Empirical Processes}

Given a function class $\G$ mapping $\mathcal{Z}$ to $\mathbb{R}$, one may consider the supremum of the empirical process:
\[
\sup\limits_{g \in \G}\left(P - P_{n}\right)g.
\]
This quantity plays an important role in statistical learning theory. Since the pioneering paper of Vapnik and Chervonenkis \cite{vapnik68}, the analysis of learning algorithms is usually performed by the tight uniform control over the process $\left(P - P_{n}\right)g$ for a special class of functions. The behaviour of the supremum of this empirical process is tightly connected with the supremum of the so-called \emph{Rademacher process}:
\[
\frac{1}{n}\E_{\varepsilon}\sup\limits_{g \in \G}\left(\sum\limits_{i = 1}^{n}\varepsilon_{i}g_{i}\right),
\]
where $g_{i}$ denotes $g(Z_{i})$, $\varepsilon_{i}$ are independent Rademacher variables taking values $\pm1$ with equal probabilities, and $\E_{\varepsilon}$ denoted the expectation over the $\varepsilon_i$ random variables (conditioning on the $Z_i$ variables). This approach, however, usually leads to suboptimal upper and lower bounds that are not capturing both improved learning rates due to the noise conditions and the localization of the complexity term.

We will instead consider different quantities, so-called \emph{shifted empirical} processes, introduced by Lecu\'e and Mitchell \cite{Lecue12}. Given $c > 0$, we consider 
\[
\sup\limits_{g \in \G}\left(P - (1 + c)P_{n}\right)g.
\]
The second important quantity is an expected supremum of the \emph{offset Rademacher process}, introduced recently by Liang, Rakhlin, and Sridharan \cite{Liang15}:
\[
\frac{1}{n}\E_{\varepsilon}\sup\limits_{g \in \G}\left(\sum\limits_{i = 1}^{n}\varepsilon_{i}g_{i}- c'g^2_{i}\right).
\]
The last quantity was introduced for the analysis of a specific aggregation procedure under the square loss and so far has not been related to a shifted process  \footnote{We should note that shifted processes and related techniques appeared independently earlier in the paper of Wegkamp \cite{Wegkamp03}. He uses the term \emph{desymmetrized} empirical processes for the shifted processes.}. In this paper, we will investigate some new properties of these processes and show how they may be applied in the classification framework. 
The following short lemma appears in a more general form in \cite{Liang15} (Lemma $5$).
\begin{lemma}
\label{expectmax}
Let $V \subset \{-1, 0, 1\}^{n}$ be a finite set of vectors of cardinality $N$. Then for any $c > 0$, 
\[
\frac{1}{n}\E_{\varepsilon}\max\limits_{v \in V}\left(\sum\limits_{i = 1}^{n}\varepsilon_{i}v_{i}- c|v_{i}|\right) \le \frac{1}{2c}\frac{\log(N)}{n}.
\]
\end{lemma}
Compare this result with an upper bound for Rademacher averages \cite{Boucheron05} where the best rate is of order $\sqrt{\frac{\log(N)}{n}}$. The next simple lemma is a new symmetrization lemma for the shifted process in expectation.
\begin{lemma}[Shifted symmetrization in expectation]
\label{symmetrization}
Let $\G$ be a functional class and $c \ge 0$ an absolute constant. Then
\[
\E\sup\limits_{g \in \G}((P - (1 + c)P_n){g}) \le \frac{c + 2}{n}\E\E_{\varepsilon}\sup\limits_{g \in \G}\left(\sum\limits_{i = 1}^{n}\varepsilon_{i}g(Z_{i})- \frac{c}{c + 2}g(Z_{i})\right).
\]
\end{lemma}
\begin{proof}
Proof technique is inspired by the proof of Theorem $3$ in \cite{Liang15}. Using standard symmetrization trick and Jensen's inequality we have
\begin{align*}
&\E\sup\limits_{g \in \G}((P - (1 + c)P_n){g})
\\
& \le \E\sup\limits_{g \in \G}\left(P'_{n}g - (1 + c)P_{n}g \right) 
\\
& = \E\sup\limits_{g \in \G}((1 + c/2)(P'_{n}g - P_{n}g) - cP'_{n}g/2 - cP_{n}g/2 )
\\
& \le 2\E\E_{\varepsilon}\sup\limits_{g \in \G}\left(\frac{1+c/2}{n}\sum\limits_{i = 1}^{n}\varepsilon_{i}g(Z_i) - cP_{n}g/2\right)  
\\
& = 2(1 + c/2)\E\E_{\varepsilon}\sup\limits_{g \in \G}\left(\frac{1}{n}\sum\limits_{i = 1}^{n}\varepsilon_{i}g(Z_i) - \frac{c/2}{1 + c/2}P_{n}g\right).
\end{align*}
\end{proof}
Interestingly, by setting $c = 0$ we immediately obtain the standard symmetrization inequality. The next lemma, which provides a novel symmetrization tool for the shifted processes in deviation requires the following definition. This result is motivated by existing classic symmetrization results \cite{Boucheron05,vapnik68}, but the proof technique is adapted for our shifted case. We say that a functional class $\G$ is a $(B, \beta)$-\emph{Bernstein class} if for any $g \in \G$ we have $P g^{2} \le B \left(P g \right)^{\beta}$. The parameter $\beta$ is called the \emph{Bernstein parameter} and $B$ the \emph{Bernstein constant}.
\begin{lemma}[Shifted symmetrization in deviation]
\label{symmetrizationdev}
Let $\G$ be a $(B, 1)$-Bernstein class, such that for all $g \in \G$ we have $Pg \ge 0$.
Fix constants $c_{1} > c_{2} > 0$. If $nt \ge \frac{B(1 + c_{2})^{2}}{c_{2}}$, then
\begin{equation*}
\Prob\left(\sup\limits_{g \in \G}(P - (1 + c_1)P_{n})g \ge t\right) \le 2\Prob\left(\sup\limits_{g \in \G}((1 + c_{2})P'_{n} - (1 + c_{1})P_{n})g \ge t/2\right).
\end{equation*}
\end{lemma}
\begin{proof}
Given a random sample let $\tilde{g}$ be the function achieving the supremum.
\begin{align*}
&\Ind[(P - (1 + c_{1})P_{n})\tilde{g} > t]\Ind[(P - (1 + c_{2})P'_{n})\tilde{g} < t/2]
\\
&\le \Ind[((1 + c_{2})P'_n - (1 + c_{1})P_n)\tilde{g} > t/2].
\end{align*}
Taking expectation with respect to the ghost sample we have
\begin{align*}
&\Ind[(P - (1 + c_{1})P_{n})\tilde{g} > t]P'[(P - (1 + c_{2})P'_{n})\tilde{g} < t/2] \le 
\\
&P'[((1 + c_{2})P'_n - (1 + c_{1})P_n)\tilde{g} > t/2].
\end{align*}
We further have
\[
P'\Bigl[(P - (1 + c_{2})P'_{n})\tilde{g} \ge t/2\Bigr] =  P'\Bigl[(P - P'_{n})\tilde{g} \ge \frac{t/2 + c_{2}P\tilde{g}}{1 + c_{2}}\Bigr].
\]
Using Chebyshev inequality together with $4ab \le (a + b)^2$ we have
\[
P'\Bigl[(P - P'_{n})\tilde{g} \ge \frac{t/2 + c_{2}P\tilde{g}}{1 + c_{2}}\Bigr]\! \le\! \frac{P\tilde{g}^2(1 + c_{2})^2}{n(t/2 + c_{2}P\tilde{g})^2}\! \le\! \frac{BP\tilde{g}(1 + c_{2})^2}{2ntc_{2}P\tilde{g}}\! =\! \frac{B(1 + c_{2})^2}{2ntc_{2}}.
\]
Finally, we have that if $\frac{ntc_2}{B(1 + c_2)^2} \ge 1$, then $P'[(P - (1 + c_{2})P'_{n})\tilde{g} < t/2]\! \ge\! \frac{1}{2}$. Taking an expectation with respect to the initial sample finishes the proof.
\end{proof}
\begin{corollary}
\label{sym}
Under conditions of the previous lemma it holds
\begin{align*}
&\Prob\left(\sup\limits_{g \in \G}(P - (1 + c_1)P_{n})g \ge t\right)
\\
&\le 4\Prob\left(\sup\limits_{g \in \G}\left(\frac{1 + c'/2}{n}\sum\limits_{i = 1}^{n}\varepsilon_{i}g(Z_i) - c'P_{n}g/2\right) \ge t'/2\right),
\end{align*}
where $c' =  \frac{c_1 - c_2}{1 + c_{2}}$ and $t' = \frac{t}{2(1 + c_2)}$
\end{corollary}
\begin{proof}
Under the notation we rewrite the result of Lemma \ref{symmetrizationdev} in the following form 
\[
\Prob\left(\sup\limits_{g \in \G}(P - (1 + c_1)P_{n})g \ge t\right) \le 2\Prob\left(\sup\limits_{g \in \G}(P'_{n} - (1 + c')P_{n})g \ge t'\right).
\]
Introducing Rademacher random variables we observe that $\sup\limits_{g \in \G}(P'_{n} - (1 + c')P_{n})g$ has the same distribution as 
\[
\sup\limits_{g \in \G}\left(\frac{1 + c'/2}{n}\sum\limits_{i = 1}^{n}\varepsilon_{i}(g_i - g'_i) - c'P_{n}g/2 - c'P'_{n}g/2\right).
\]
Now we have
\begin{align*}
&\sup\limits_{g \in \G}\left(\frac{1 + c'/2}{n}\sum\limits_{i = 1}^{n}\varepsilon_{i}(g_i - g'_i) - c'P_{n}g/2 - c'P'_{n}g/2\right) 
\\
&\le \sup\limits_{g \in \G}\left(\frac{1 + c'/2}{n}\sum\limits_{i = 1}^{n}\varepsilon_{i}g_i - c'P_{n}g/2\right) + \sup\limits_{g \in \G}\left(-\frac{1 + c'/2}{n}\sum\limits_{i = 1}^{n}\varepsilon_{i}g'_i  - c'P'_{n}g/2\right).
\end{align*}
Observe that both summands have the same distribution. The claim easily follows.
\end{proof}

Let $\mathbf{s}$ be the star number of a class of binary classifiers $\F$. Hanneke \cite{Hanneke15a} recently proved that in this case
\begin{equation}
\label{starlabel}
\E P_{X}(\dis(\mathcal{V}_n)) \le \frac{\mathbf{s}}{n + 1}, 
\end{equation}
where $\mathcal{V}_n = \{f \in \F|P_{n}[f(X) \neq f^{*}(X)] = 0\}$ is the \emph{version space}.  That work also established a similar result holding with high probability: with probability at least $1 - \delta$,
\begin{equation}
\label{starlabelwhp}
P_{X}(\dis(\mathcal{V}_n)) \le \frac{21\mathbf{s}}{n} + \frac{16\log(\frac{3}{\delta})}{n}. 
\end{equation}
This result means that if the star number is bounded, then in the realizable case the expected measure of disagreement of the version space has order $\frac{\mathbf{s}}{n}$, where $n$ is the size of the learning sample. A reader familiar with the work of Haussler, Littlestone, and Warmuth \cite{Haussler94} may remember that the performance of some learning algorithms can be controlled by the maximum possible out-degree in a corresponding orientation of the data-induced \emph{one-inclusion graph}, and that there exists such an orientation with maximum out-degree at most the VC dimension.  The relation of this to the present context is that (as noted by \cite{Hanneke15}) 
the star number can equivalently be defined as the largest possible value of the 
(undirected) degree of a data-induced one-inclusion graph.  Thus, instead of the 
\emph{out-degree} of an oriented data-induced one-inclusion graph, the measure of the region of 
disagreement is controlled by the largest possible value of the \emph{(undirected) degree} 
of the data-induced one-inclusion graph.

Since both the ERM and the optimal classifier are contained in $\mathcal{V}_n$ in the realizable case, one consequence of the above results is that, when $\mathbf{s} \approx d$, ERM achieves the optimal order $d/n$ in its error rate. This happens, for example, in the case of threshold classifiers. Even more interesting, \cite{Hanneke15a} used the bound \eqref{starlabelwhp} in a more subtle way to show that ERM in the realizable case obtains expected error rate of order $\frac{d}{n}\log\frac{n \land \mathbf{s}}{d}$, and with probability at least $1-\delta$ has error rate bounded as in \eqref{s-over-d-bound}: i.e., of order $\frac{d}{n}\log \frac{n \land \mathbf{s}}{d} + \frac{1}{n}\log\frac{1}{\delta}$.
Via a more sophisticated variant of this argument, we obtain the following theorem, which is one of the novel contributions of this work.  It offers interesting general refinements over \eqref{s-over-d-bound} which we discuss below.  Its proof is included in the appendix.
\begin{theorem}
\label{startheorem}
Let $\mathbf{s}$ be the star number of a class of binary classifiers $\F$. In the realizable case, any ERM $\hat{f}$ has
\[
\E R(\hat{f}) \lesssim \frac{\log\left(\mathcal{S}_{\mathcal{F}}\left(\mathbf{s} \wedge n\right)\right)}{n}.
\]
Moreover, with probability at least $1 - \delta$,
\[
R(\hat{f}) \lesssim \frac{\log\left(\mathcal{S}_{\mathcal{F}}\left(\mathbf{s} \wedge n\right)\right)}{n} + \frac{\log(\frac{1}{\delta})}{n}.
\]
\end{theorem}

We may prove (due to Vapnik and Chervonenkis's bound on the growth function \cite{vapnik68}) that this inequality is an alternative way of recovering the upper bound 
\eqref{s-over-d-bound} discussed above, and its implied bound 
$\E R(\hat{f}) \lesssim \frac{d\log\left(\frac{\mathbf{s} \wedge n}{d}\right)}{n}$
for ERM,
also established by \cite{Hanneke15a}.
\begin{example}
\label{twoclassesexample}
Theorem \ref{startheorem} yields simple examples showing the gaps in the 
distribution-free bound \eqref{s-over-d-bound} in the realizable case.  
Specifically, suppose $\mathcal{X} = \{x_{1}, \ldots, x_{\mathbf{s}}\}$, 
define class $\F_{1}$ as the classifiers on this $\mathcal{X}$ with at most $d$ 
points classified $1$, and class $\F_{2}$ as the classifiers having at most $d-1$ points 
classified $1$ among $\{x_1,\ldots,x_{d-1}\}$ and at most one point classified $1$ among 
$\{x_{d},\ldots,x_{\mathbf{s}}\}$.  For both $\F_1$ and $\F_2$, the VC dimension is $d$ 
and the star number is $\mathbf{s}$.  However, for $\F_1$ Theorem~\ref{startheorem} 
gives a bound of order $\frac{d\log\left(\frac{\mathbf{s} \wedge n}{d}\right)}{n}$, 
but for $\F_{2}$ it gives a smaller bound of order 
$\frac{d + \log\left(\mathbf{s} \wedge n\right)}{n}$.  
In both cases, these are known to be tight characterizations of ERM in the realizable case 
\cite{Haussler94,Hanneke15a}.  

It should be noted, however, that one can also construct 
examples where Theorem~\ref{startheorem} is itself not tight.
For instance, for $\mathbf{s} > 2(d-1)$, consider  
$\mathcal{X} = [0,d-1) \cup \{d,\ldots,\mathbf{s}-d+1\}$
and
$\F_{3}$ as the functions that classify as $1$ points in a set 
$\bigcup_{i=1}^{d-1} [t_{i},i)$, for some parameters $t_{i} \in [i-1,i)$,
and also classify as $1$ at most one point among $\{d,\ldots,\mathbf{s}-d+1\}$,
and classify all other points in $\mathcal{X}$ as $-1$. 
The VC dimension of $\F_{3}$ is $d$ and the star number is $\mathbf{s}$.
Theorem~\ref{startheorem} yields a bound of order 
$\frac{d \log\left(\frac{\mathbf{s} \wedge n}{d}\right)}{n}$,
whereas one can easily verify that for this $\F_{3}$ ERM (in the realizable case) actually 
achieves an expected risk of order $\frac{d + \log\left(\mathbf{s} \wedge n\right)}{n}$.
\end{example} 

\section{Bounds in Terms of a Global Packing}
The main aim of this section is to give a simple bound in terms of a fixed point of global packings. We will further significantly improve this result in the next section, and therefore for simplicity here we will consider only the realizable case. We note that a similar result may be derived from classic results on ratio type empirical processes (see Section $19.6$ of \cite{Anthony99}). We include the details of our proof here anyway, as it also serves to illustrate certain aspects of our approach in simplified form. 

Given a set of $n$ points we define for any two $f, g \in \F$ where $\rho_H(f,g) = |\{ i \in \{1,\ldots,n\} : f(x_i) \neq g(x_i) \}|$. 
We further introduce
\[
\mathcal{M}^{*}_{1}(\F, \gamma, n) = \max\limits_{x_{1}, \ldots, x_{n} \in \mathcal{X}}\mathcal{M}_{1}(\F(\{x_{1}, \ldots, x_{n}\}), \gamma),
\]
where $\mathcal{M}_{1}(\mathcal{H},\varepsilon)$ denotes the size of a maximal $\varepsilon$-packing of $\mathcal{H}$ under $\rho_H$ distance (for the given $x_1,\ldots,x_n$ points) and $\F(\{x_{1}, \ldots, x_{n}\})$ is a set of projections of $\F$ on $\{x_{1}, \ldots, x_{n}\}$.

In many statistical frameworks optimal rates are usually obtained when one carefully balances the radius and the logarithm of a packing number with respect to the same radius (for example, Yang and Barron \cite{Yang99}). It will be shown that in our bounds it is natural to choose $\gamma$ such that $c\gamma \approx \log(\mathcal{M}^{*}_{1}(\F, \gamma, n))$ for some $c \in [0, 1]$. So we define 
\[
\gamma^{*}_{c}(n, \F) = \max\{\gamma \in \mathbb{N}: c\gamma \le \log\left(\mathcal{M}^{*}_{1}(\F, \gamma, n)\right)\}.
\]
The value $\gamma^{*}_{c}(n, \F)$ will be referred to as a \emph{fixed point of empirical entropy}. When $\F$ is clear from the context, we simply write $\gamma^{*}_{c}(n)$ instead of $\gamma^{*}_{c}(n, \F)$.  Note that $\gamma^{*}_{c}(n,\F)$ is a well-defined strictly positive-valued quantity, since we are using the truncated logarithm.

\begin{proposition}
\label{coveringbound}
Fix any function class $\F$; denote its VC dimension $d$.  If $P \in \mathcal{P}(1, \F)$ (realizable case), then for any ERM $\hat{f}$,
\[
\E R(\hat{f}) \lesssim \frac{\gamma^{*}_{\frac{1}{2}}(n)}{n}.
\]
Moreover with probability at least $1 - \delta$,
\[
R(\hat{f})  \lesssim \frac{\gamma^{*}_{\frac{1}{2}}(n)}{n} + \frac{\log{\frac{1}{\delta}}}{n},
\]
and
\begin{equation}
\label{fixedpointupperbound}
\gamma^{*}_{\frac{1}{2}}(n) \lesssim d\log(n/d).
\end{equation}
\end{proposition}
To prove this proposition we need a technical lemma, which may be considered as a modification of Lemma $6$ in \cite{Liang15}.
\begin{lemma}
\label{expectmaxnew}
Let $\G$ be a set of functions taking binary values, and let $c \in [0, 1]$ be a constant. Let $\varepsilon_{1},\ldots,\varepsilon_{n}$ be independent Rademacher random variables. Then
\[
\frac{1}{n}\E_{\varepsilon}\max\limits_{g \in \G}\left(\sum\limits_{i = 1}^{n}\varepsilon_{i}g(X_{i})- cg(X_{i})\right) \le \frac{7\gamma^{*}_{c}(n)}{n}.
\]
\end{lemma}
\begin{proof}
Given $X_1,\ldots,X_n$, let $V = \{ (g(X_1),\ldots,g(X_n)) : g \in \G \}$ denote the set of binary vectors corresponding to the values of functions in $\G$.
As above, for a fixed $\gamma$ and fixed minimal $\gamma$-covering subset $\mathcal{N}_{\gamma} \subseteq V$, for each $v \in V$, $p(v)$ will denote the closest vector to $v$ in $\mathcal{N}_{\gamma}$. First we follow the decomposition proposed by Liang, Rakhlin, and Sridharan \cite{Liang15}: 
\begin{align*}
&\E_{\varepsilon}\max\limits_{v \in V}\left(\sum\limits_{i = 1}^{n}\varepsilon_{i}v_{i}- cv_{i}\right)
\\
&\le
\E_{\varepsilon}\max\limits_{v \in V}\left(\sum\limits_{i = 1}^{n}\varepsilon_{i}\left(v_{i}- p(v)_{i}\right)\right) + 
\max\limits_{v \in V}\left(\sum\limits_{i = 1}^{n}\frac{c}{4}p(v)_{i}- cv_{i}\right)
\\
&+
\E_{\varepsilon}\max\limits_{v \in V}\left(\sum\limits_{i = 1}^{n}\varepsilon_{i}p(v)_{i}- \frac{c}{4}p(v)_{i}\right).
\end{align*}
Since $p(v)$ is within Hamming distance $\gamma$ of $v$, we know $\sum_{i=1}^{n} p(v)_{i} \leq \gamma + \sum_{i=1}^{n} v_{i}$, and therefore the second summand in the above expression is at most 
\[\max\limits_{v \in V} \left( \frac{c}{4}\gamma - \frac{3 c}{4}  \sum_{i=1}^{n} v_{i} \right) \leq \frac{c}{4} \gamma.\]  The third summand is upper bounded by $\frac{2}{c}\log(|\mathcal{N}_{\gamma}|)$ by Lemma \ref{expectmax} and the first term is upper bounded by $\gamma$ by the $\gamma$-cover property of the $p(v)$ vectors. Then we use the standard relation that the size a of minimal covering is less than or equal to the size of a maximal packing \cite{Devr01} to conclude that
\[
\frac{1}{n}\E_{\varepsilon}\max\limits_{v \in V}\left(\sum\limits_{i = 1}^{n}\varepsilon_{i}v_{i}- cv_{i}\right) \le \frac{(1+c/4)\gamma}{n} + \frac{2}{c}\frac{\log(\mathcal{M}_{1}(V, \gamma))}{n}.
\]
By choosing $\gamma = \gamma^{*}_{c}(n) + 1$ we have
\[
\frac{(1+c/4)\gamma}{n} + \frac{2}{c}\frac{\log(\mathcal{M}_{1}(V, \gamma))}{n} \le \frac{(1+c/4)(\gamma^{*}_{c}(n) + 1)}{n} + \frac{2(\gamma^{*}_{c}(n) + 1)}{n} \le \frac{7\gamma^{*}_{c}(n)}{n}.
\] 
\end{proof}
\begin{proof}[Proposition \ref{coveringbound}]
First we introduce a \emph{loss class} $\mathcal{G}_{f^{*}} = \{x \to \Ind[f(x)\neq f^{*}(x)]\ :\ f \in \F\}$. Let $\hat{f}$ be any ERM and $\hat{g}$ be a corresponding function in the loss class $\mathcal{G}_{f^{*}}$. We obviously have $\E R(\hat{f}) = P\hat{g}$ and $P_{n}\hat{g} = 0 $. Then for any $c > 0$
\[
\E R(\hat{f}) = \E(R(\hat{f}) - (1 + c)R_{n}(\hat{f})) \le \E\sup\limits_{g \in \mathcal{G}_{f^{*}}}(Pg - (1+c)P_{n}g).
\]
By Lemma \ref{symmetrization} we have
\[
\E\sup\limits_{g \in \G_{f^{*}}}(Pg - (1+c)P_{n}g) \le \frac{c + 2}{n}\E\E_{\varepsilon}\sup\limits_{g \in \G_{f^{*}}}\left(\sum\limits_{i = 1}^{n}\varepsilon_{i}g(X_{i})- \frac{c}{c + 2}g(X_{i})\right)
\]
Applying the Lemma \ref{expectmaxnew} and fixing $c = 2$ we finish the proof of the bound on the expectation.  The high probability version of this bound is deferred to the appendix.
\end{proof}
\begin{example}
\label{threshold}
Consider the class of threshold classifiers, that is $\F = \{x \to 2\Ind[x \le t] - 1: t \in \mathbb{R}\}$. Using the definition of the star number it easy to see that it is equal to $2$ in this case and Theorem \ref{startheorem} gives an optimal $\frac{1}{n}$ upper bound for ERM. At the same time the worst case packing numbers $\mathcal{M}^{*}_{1}(\F, \gamma, n)$ are of order $\frac{n}{\gamma}$. A simple analysis of the fixed point gives us $\gamma^{*}_{\frac{1}{2}}(n) \simeq \log(n)$ and thus Proposition \ref{coveringbound} will give us suboptimal $\frac{\log{n}}{n}$ distribution free upper bound. Although we captured that the rate is faster than $\frac{1}{\sqrt{n}}$, our analysis of the complexity term is suboptimal.  The next section discusses a correction for this, which also yields optimal rates under moderate bounded noise in general.
\end{example}

\section{Local Metric Entropy}
This section presents our main result.  Toward this end, we introduce a new complexity measure: the \emph{worst-case local empirical packing numbers}. Given a set of $n$ points we fix some $f \in \F$ and construct a Hamming ball of the radius $\gamma$. So, $\mathcal{B}_{H}(f, \gamma, \{x_{1}, \ldots, x_{n}\}) = \{g \in \F| \rho_H(f, g) \le \gamma\}$ and define
\begin{equation}
\label{localentropy}
\mathcal{M}^{\text{loc}}_{1}(\F, \gamma, n, h) = \max\limits_{x_{1}, \ldots, x_{n}}\max\limits_{f \in \F}\max\limits_{\varepsilon \ge \gamma}\mathcal{M}_{1}(\mathcal{B}_{H}(f, \varepsilon/h, \{x_{1}, \ldots, x_{n}\}), \varepsilon/2),
\end{equation}
where once again $\mathcal{M}_{1}(\mathcal{H},\varepsilon)$ denotes the size of a maximal $\varepsilon$-packing of $\mathcal{H}$ under $\rho_H$ distance (for the given $x_1,\ldots,x_n$ points).
Fix any $h, h' \in (0, 1]$ and define
\[
\gamma^{\text{loc}}_{h, h'}(n, \F) =  \max\{\gamma \in \mathbb{N}: h\gamma \le \log(\mathcal{M}^{\text{loc}}_{1}(\F, \gamma, n, h'))\}.
\]
When $\F$ is clear from the context, we simply write $\gamma^{\text{loc}}_{h, h'}(n)$ instead of $\gamma^{\text{loc}}_{h, h'}(n, \F)$. The quantity $\gamma^{\text{loc}}_{h,h'}(n)$ defines the \emph{fixed point of a local empirical entropy}. We note that, because $1 \leq d < \infty$ in this work, when $h,h' > 0$ the set on the right in this definition is finite and nonempty,
so that $\gamma^{\text{loc}}_{h,h'}(n)$ is a well-defined strictly-positive integer.  
Indeed, for any $h,h' \in (0,1]$, the value $\gamma = \lfloor \frac{1}{h} \rfloor$ satisfies $h\gamma \leq 1$, so that (because $\log(\cdot)$ is the truncated logarithm) this $\gamma$ is contained in the set; in particular, this implies $h \gamma^{\text{loc}}_{h,h'}(n,\F) \geq h \lfloor \frac{1}{h} \rfloor \geq \frac{1}{2}$ always.

The next theorem is the main upper bound of this paper.
\begin{theorem}
\label{mainupperbound}
Fix any function class $\F$; denote its VC dimension $d$ and star number $\mathbf{s}$.  Fix any $h \in \left(\sqrt{\frac{d}{n}}, 1\right]$. 
If $P \in \mathcal{P}(h, \F)$, then for any ERM $\hat{f}$,
\begin{equation}
\label{expectbound}
\E(R(\hat{f}) - R(f^{*})) \lesssim \frac{\gamma^{\text{\emph{loc}}}_{h, h}(n)}{n}. 
\end{equation}
Also, with probability at least $1-\delta$,
\begin{equation}
\label{devbound}
R(\hat{f}) - R(f^{*}) \lesssim \frac{\gamma^{\text{\emph{loc}}}_{h, h}(n)}{n} + \frac{\log(\frac{1}{\delta})}{nh}.
\end{equation}
Moreover
\begin{equation}
\label{controloffixedpoint}
 \frac{d + \log\left(nh^2 \wedge \mathbf{s}\right)}{h} \lesssim \gamma^{\text{\emph{loc}}}_{h, h}(n) \lesssim \frac{d\log\left(\frac{nh^2}{d} \wedge \mathbf{s}\right)}{h} + \frac{d\log\left(\frac{1}{h}\right)}{h}. 
\end{equation}
\end{theorem} 

Our complexity term \eqref{controloffixedpoint} is not worse than the distribution-free upper bound \eqref{Ginecorollary} implied by the bound \eqref{koltchbound} of Gin\'e and Koltchinskii when $h$ is bounded from $0$ by a constant. In the last section we will discuss potential suboptimality when $h$ is small, due to the term $\frac{d\log\left(\frac{1}{h}\right)}{h}$ in \eqref{controloffixedpoint}. Another interesting property is that the bounds \eqref{expectbound} and \eqref{devbound} involve neither the VC dimension nor the star number explicitly. At the same time one can control the complexity term with both of them from below and above. 

For any given $f \in \F$, denote $g_{f}(x,y) = \Ind[f(x)\neq y] - \Ind[f^{*}(x)\neq y]$.
Consider the \emph{excess loss class} 
$\mathcal{G}_{\mathcal{Y}} = \{ g_{f} | f \in \F\}$, the class $\mathcal{G}_{f^{*}} = \{x \to \Ind[f(x)\neq f^{*}(x)]\ |\ f \in \F\}$
and the class $\F^{*} = \frac{1}{2}(\F - f^*)$. The last class consists of functions of the form $\frac{1}{2}(f - f^*)$ for $f \in \F$. The following properties are well known.
\begin{enumerate}
\item For any $g_{f} \in \mathcal{G}_{\mathcal{Y}}$ it holds $g_{f}^2(x, y) = \Ind[f(x) \neq f^{*}(x)] = \frac{1}{2}|f(x) - f^{*}(x)| = \frac{1}{4}(f(x) - f^{*}(x))^2$.
\item For any $g_{f} \in \mathcal{G}_{\mathcal{Y}}$ it holds $g_{f}(x, y) = \frac{y(f^{*}(x) - f(x))}{2}$.
\item For any  $P \in \mathcal{P}(h, \F)$ the class $\mathcal{G}_{\mathcal{Y}}$ is a $(\frac{1}{h}, 1)$-Bernstein class \cite{Boucheron05} and $R(f^{*}) \le \frac{1}{2}(1 - h)$ \cite{Devr95}.
\end{enumerate}

\begin{lemma}[Contraction]
\label{contraction}
Let $\G_{\mathcal{Y}}$ be an excess loss class associated with a given class $\F$, and fix any $h \in [0, 1]$. For any $c \in [0, 1]$ and any distribution $P \in \mathcal{P}(h, \F)$ we have conditionally on $X_1,\ldots,X_n$
\begin{align*}
&\E_{Y|X}\E_{\varepsilon}\sup\limits_{g \in \G_{\mathcal{Y}}}\left(\sum\limits_{i = 1}^{n}\varepsilon_{i}g(X_{i}, Y_{i})- cg(X_{i}, Y_{i})\right) 
\\
&\le \E_{\varepsilon}\sup\limits_{f'\in \F^*}\left(\sum\limits_{i = 1}^{n}\varepsilon_{i}f'(X_i) - \frac{1}{2}hc|f'(X_i)|\right)
\\
&\quad+\frac{3c}{2}\E_{\xi}\sup\limits_{g'\in \G_{f^*}}\left(\sum\limits_{i = 1}^{n}\xi_{i}g'(X_i) - \frac{1}{3}hg'(X_i)\right)
\end{align*}
where $\xi_{1}, \ldots, \xi_{n}$ are random variables conditionally independent given $X_1,\ldots,X_n$, with $\E[ \xi_i | X_1,\ldots,X_n ] = 0$ and $\E[\exp(\lambda\xi_{i})|X_1,\ldots,X_n] \le \exp(\frac{\lambda^{2}}{2})$ for all $\lambda$. Moreover, for all $x > 0$
\begin{align*}
&\Prob_{Y|X, \varepsilon}\left(\sup\limits_{g \in \G_{\mathcal{Y}}}\left(\sum\limits_{i = 1}^{n}\varepsilon_{i}g(X_{i}, Y_{i})- cg(X_{i}, Y_{i})\right) 
 \ge x\right) 
\\
&\le \Prob_{\varepsilon}\left(\sup\limits_{f'\in \F^*}\left(\sum\limits_{i = 1}^{n}\varepsilon_{i}f'(X_i) - \frac{1}{2}hc|f'(X_i)|\right) \ge \frac{x}{2}\right)
\\
&\quad+ \Prob_{\xi}\left(\sup\limits_{g'\in \G_{f^*}}\left(\sum\limits_{i = 1}^{n}\xi_{i}g'(X_i) - \frac{1}{3}hg'(X_i)\right) \ge \frac{x}{3c}\right)
\end{align*}
\end{lemma}
\begin{proof}
First we notice that any $g \in \mathcal{G}_{\mathcal{Y}}$ may be defined by some $f \in \F$.
\begin{align*}
&\E_{Y|X}\E_{\varepsilon}\sup\limits_{g \in \G_{\mathcal{Y}}}\left(\sum\limits_{i = 1}^{n}\varepsilon_{i}g(X_{i}, Y_{i})- cg(X_{i}, Y_{i})\right)
\\
&= \E_{Y|X}\E_{\varepsilon}\sup\limits_{f \in \mathcal{F}}\left(\sum\limits_{i = 1}^{n}\frac{1}{2}\varepsilon_{i}Y_{i}(f(X_i) - f^{*}(X_i)) - cg_{f}(X_{i}, Y_{i})\right)
\\
&= \E_{Y|X}\E_{\varepsilon}\sup\limits_{f \in \mathcal{F}}\left(\sum\limits_{i = 1}^{n}\frac{1}{2}\varepsilon_{i}(f(X_i) - f^{*}(X_i)) - cg_{f}(X_{i}, Y_{i})\right)
\\
&= \frac{1}{2}\E_{Y|X}\E_{\varepsilon}\sup\limits_{f \in \mathcal{F}}\left(\sum\limits_{i = 1}^{n}\varepsilon_{i}(f(X_i) - f^{*}(X_i)) - 2cg_{f}(X_{i}, Y_{i})\right).
\end{align*}
Now consider the term $-\sum\limits_{i = 1}^{n}g(X_{i}, Y_{i})$.  Denoting $h'_i = 1-2P(f^{*}(X_i) \neq Y_i | X_i)$ (an $X_i$-dependent random variable), we know that $1 \ge h'_i \ge h$ almost surely.  Furthermore, the event that $f^{*}(X_i) \neq Y_i$ has conditional probability (given $X_i$) equal $\frac{1}{2}(1-h'_i)$, and on this event we have $\frac{1}{2}|f(X_i) - f^*(X_i)| = -g(X_{i}, Y_{i})$.  Similarly, the event that $f^{*}(X_i) = Y_i$ occurs with conditional probability (given $X_i$) equal $\frac{1}{2}(1+h'_i)$, and on this event we have $\frac{1}{2}|f(X_i) - f^*(X_i)| = g(X_{i}, Y_{i})$.  Thus, defining $\xi^{(h')}_{i} = h'_i + \Ind[f^{*}(X_i) \neq Y_i] - \Ind[f^{*}(X_i) = Y_i]$, these $\xi^{(h')}_1,\ldots,\xi^{(h')}_n$ random variables are conditionally independent given $X_1,\ldots,X_n$, with $\E[ \xi^{(h')}_{i} | X_1,\ldots,X_n ] = 0$.  
In particular, if $h'_i = 0$ for all $i$, these are Rademacher random variables, while if $h'_{i} = 1$ these random variables are equal to $0$ with probability $1$.  
Now note that, by the above reasoning about these events, 
\begin{align*}
-\sum\limits_{i = 1}^{n}g(X_{i}, Y_{i}) &= -\sum\limits_{i=1}^{n} \frac{h'_i}{2}|f(X_i) - f^*(X_i)|  + \sum\limits_{i=1}^{n} \frac{\xi^{(h')}_{i}}{2}|f(X_i) - f^*(X_i)| 
\\
&\le -(\min\limits_{i}h'_{i})\sum\limits_{i = 1}^{n}\frac{1}{2}|f(X_i) - f^*(X_i)| + \sum\limits_{i = 1}^{n}\frac{\xi^{(h')}_{i}}{2}|f(X_i) - f^*(X_i)|.
\end{align*}
Using the fact that $h \le h'_i$ almost surely, we have
\begin{align*}
&\frac{1}{2}\E_{Y|X}\E_{\varepsilon}\sup\limits_{f \in \F}\left(\sum\limits_{i = 1}^{n}\varepsilon_{i}(f(X_i) - f^*(X_i)) - 2cg_{f}(X_{i}, Y_{i})\right)
\\
&\le\E_{\xi}\E_{\varepsilon}\sup\limits_{f'\in \F^*}\left(\sum\limits_{i = 1}^{n}\varepsilon_{i}f'(X_i) + c\xi^{(h')}_{i}|f'(X_i)| - hc|f'(X_i)|\right)
\\
&\le\E_{\varepsilon}\sup\limits_{f'\in \F^*}\left(\sum\limits_{i = 1}^{n}\varepsilon_{i}f'(X_i) - \frac{1}{2}hc|f'(X_i)|\right) 
\\
&\quad+c\E_{\xi}\sup\limits_{f'\in \F^*}\left(\sum\limits_{i = 1}^{n}\xi^{(h')}_{i}|f'(X_i)| - \frac{1}{2}h|f'(X_i)|\right).
\end{align*}
Finally, we have $-1 \le \xi^{(h_i)}_i \le 2$, Hoeffding's lemma (\cite{Devr95} Lemma 8.1) implies $\E[\exp(\lambda\xi^{(h_i)}_i)|X_1,\ldots,X_n] \le \exp(9\lambda^{2}/8)$. The first claim of the Lemma easily follows, taking $\xi_i = \frac{2}{3}\xi^{(h_i)}_i $.

For the proof of the second claim we repeat almost the same steps. Observe that
$\sup\limits_{g \in \G_{\mathcal{Y}}}\left(\sum\limits_{i = 1}^{n}\varepsilon_{i}g(X_{i}, Y_{i})- cg(X_{i}, Y_{i})\right)$
has the same distribution (given $X_{1}, \ldots, X_n$) as 
$\frac{1}{2}\sup\limits_{f \in \mathcal{F}}\left(\sum\limits_{i = 1}^{n}\varepsilon_{i}(f(X_i) - f^{*}(X_i)) - 2cg_{f}(X_{i}, Y_{i})\right).$
Finally, using the definition of $\xi_{i}$ we have almost surely (once again given $X_{1}, \ldots, X_n$)
\begin{align*}
&\frac{1}{2}\sup\limits_{f \in \mathcal{F}}\left(\sum\limits_{i = 1}^{n}\varepsilon_{i}(f(X_i) - f^{*}(X_i)) - 2cg_{f}(X_{i}, Y_{i})\right)
\\
&\le\sup\limits_{f'\in \F^*}\left(\sum\limits_{i = 1}^{n}\varepsilon_{i}f'(X_i) - \frac{1}{2}hc|f'(X_i)|\right) 
\\
&\quad+c\E_{\xi}\sup\limits_{f'\in \F^*}\left(\sum\limits_{i = 1}^{n}\xi^{(h')}_{i}|f'(X_i)| - \frac{1}{2}h|f'(X_i)|\right).
\end{align*}
The second claim of the Lemma follows.
\end{proof}
Recall that $\mathcal{G}_{f^{*}} = \{x \to \Ind[f(x)\neq f^{*}(x)]\ |\ f \in \F\}$
and $\F^{*} = \frac{1}{2}(\F - f^*)$.
\begin{lemma}[Localization]
\label{localization}
Given the class of classifiers $\F$ let $\G = \F^*$ or $\G = \G_{f^*}$, and let $c \in [0, \frac{1}{4}]$ be a constant. Let $\xi_1,\ldots,\xi_n$ be any random variables conditionally independent given $X_1,\ldots,X_n$, with $|\xi_i| \lesssim 1$, and with $\E[ \xi_i | X_1,\ldots,X_n ] = 0$ and $\E[ \exp(\lambda \xi_i) | X_1,\ldots,X_n ] \leq \exp(\frac{\lambda^2}{2})$ for all $\lambda$. Then if $\G$ contains the zero function 
\[
\frac{1}{n}\E_{\xi}\sup\limits_{g \in \G}\left(\sum\limits_{i = 1}^{n}\xi_{i}g(X_{i})- 4c|g(X_{i})|\right) \lesssim \frac{\gamma^{\text{\emph{loc}}}_{c, c}(n, \F)}{n}.
\]
\end{lemma}

The proof of this lemma is deferred to the appendix.
We are now ready for the proof of Theorem~\ref{mainupperbound}.

\begin{proof}[Theorem \ref{mainupperbound}]
Let $\hat{f}$ be an ERM and $\hat{g}$ be a corresponding function in the excess loss class $\G_{\mathcal{Y}}$. We obviously have $\E(R(\hat{f}) - R(f^{*})) = \E P\hat{g}$ and $P_{n}\hat{g} \le 0$. Then for any $c > 0$,
\[
\E(R(\hat{f}) - R(f^{*})) \le \E(P\hat{g} - (1 + c)P_{n}\hat{g}) \le \E\sup\limits_{g \in \G_{\mathcal{Y}}}(Pg - (1+c)P_{n}g).
\]
Now using the symmetrization lemma (Lemma \ref{symmetrization}) we have
\[
\E\sup\limits_{g \in \G_{\mathcal{Y}}}(Pg - (1+c)P_{n}g) \le \frac{c + 2}{n}\E\E_{\varepsilon}\sup\limits_{g \in \G_{\mathcal{Y}}}\left(\sum\limits_{i = 1}^{n}\varepsilon_{i}g(X_{i}, Y_{i})- \frac{c}{c + 2}g(X_{i}, Y_{i})\right).
\]
Applying the contraction lemma (Lemma \ref{contraction})
\begin{align*}
&\frac{c + 2}{n}\E\E_{\varepsilon}\sup\limits_{g \in \G_{\mathcal{Y}}}\left(\sum\limits_{i = 1}^{n}\varepsilon_{i}g(X_{i}, Y_{i})- \frac{c}{c + 2}g(X_{i}, Y_{i})\right) 
\\
&\le
\frac{(c + 2)}{n}\E\E_{\varepsilon}\sup\limits_{f'\in \F^*}\left(\sum\limits_{i = 1}^{n}\varepsilon_{i}f'(X_i) - \frac{hc}{2(c + 2)}|f'(X_i)|\right)
\\
&\quad+
\frac{3c}{2n}\E\E_{\xi}\sup\limits_{f'\in \F^*}\left(\sum\limits_{i = 1}^{n}\xi_{i}|f'(X_i)| - \frac{1}{3}h|f'(X_i)|\right).
\end{align*}
We are ready to apply the localization lemma (Lemma \ref{localization}). 
%
The conditions on the $\xi_i$ and $\varepsilon_i$ variables required for Lemma~\ref{localization} are supplied by Lemma~\ref{contraction}, and all functions in $\F^*$ take only $\{-1, 0 , 1\}$ values. Thus, for a fixed $c$,
\[
\frac{(c + 2)}{n}\E\E_{\varepsilon}\sup\limits_{f'\in \F^*}\left(\sum\limits_{i = 1}^{n}\varepsilon_{i}f'(X_i) - \frac{hc}{2(c + 2)}|f'(X_i)|\right)  \lesssim \frac{\gamma^{\text{loc}}_{h, h}(n)}{n}.
\]
The same bound holds for $\frac{3c}{2n}\E\E_{\xi}\sup\limits_{f'\in \F^*}\left(\sum\limits_{i = 1}^{n}\xi_{i}|f'(X_i)| - \frac{1}{3}h|f'(X_i)|\right)$.
The proof of the deviation bound is analogous, and is presented in the appendix.  The claimed bounds on $\gamma^{\text{loc}}_{h,h}(n)$ are established in Proposition~\ref{fixedcontrol} below.
\end{proof}
The following proposition finishes the proof of Theorem \ref{mainupperbound}.
\begin{proposition}
\label{fixedcontrol}
Let $d$ be the VC-dimension and $\mathbf{s}$ be the star number of $\F$. For any $h \in (0, 1]$, it holds 
\[
 \frac{d + \log\left(nh^2 \wedge \mathbf{s}\right)}{h} \wedge \sqrt{dn} \lesssim \gamma^{\text{\emph{loc}}}_{h, h}(n) \lesssim \frac{d\log\left(\frac{nh^2}{d} \wedge \mathbf{s}\right)}{h} + \frac{d\log(\frac{1}{h})}{h}. 
\]
\end{proposition}
\begin{proof}
The first part of the proof closely follows the proof of Theorem $17$ in \cite{Hanneke15}, with slight modifications, to arrive at an upper bound on $\mathcal{M}^{\text{loc}}_{1}(\F, \gamma, n, h)$. The suprema in the definition of local empirical entropy are achieved at some set $\{x_{1}, \ldots, x_{n}\}$, some function $f \in \F$, and some $\varepsilon \in [\gamma,n]$.
Letting $r = \varepsilon/n$, denote by $\mathcal{M}_r$ the maximal $(rn/2)$-packing (under $\rho_H$) of $\mathcal{B}_{H}(f,rn/h,\{x_1,\ldots,x_n\})$, so that $|\mathcal{M}_r| = \mathcal{M}^{\text{loc}}_{1}(\F,\gamma,n,h)$.  Also introduce a uniform probability measure $P_X$ on $\{x_1,\ldots,x_n\}$ and fix $m = \left\lceil\frac{4}{r}\log(|\mathcal{M}_r|)\right\rceil$. Let $X_{1}, \ldots, X_{m}$ be $m$ independent $P_X$-distributed random variables, and let $A$ denote the event that, for all $g,g' \in \mathcal{M}_{r}$ with $g \neq g'$, there exists an $i \in \{1, \ldots, n\}$ such that $g(X_{i}) \neq g'(X_{i})$. For a given pair of distinct functions $g,g' \in \mathcal{M}_r$, they disagree on some $X_i$ with probability
\[
1 - (1 - P_X(g(X) \neq g'(X)))^m  > 1 - \exp(-rm/2) \ge 1 - \frac{1}{|\mathcal{M}_r|^2}.
\]
Using a union bound and summing over all possible unordered pairs $g, g' \in \mathcal{M}_r$ will give us that $\mathbb{P}(A) > \frac{1}{2}$. On the event $A$, functions in $\mathcal{M}_{r}$ realize distinct classifications of $X_{1}, \ldots, X_{m}$. For any 
\[X_i \notin \dis(\mathcal{B}_{H}(f, rn/h, \{x_{1}, \ldots, x_{n}\}),
\] all classifiers in $\mathcal{M}_{r}$ agree. Thus, $|\mathcal{M}_r|$ is bounded by the number of different classifications $\{X_{1}, \ldots, X_{m}\} \cap \dis(\mathcal{B}_{H}(f, rn/h))$ realized by classifiers in $\F$. By the multiplicative Chernoff bound 
(see \cite{motwani:95}, Section 4.1), 
on an event $B$ with $\mathbb{P}(B) \ge \frac{1}{2}$ we have $
|\{X_{1}, \ldots, X_{m}\} \cap \dis(\mathcal{B}_{H}(f, rn/h))|  \le 
1 + 2eP_X(\dis(\mathcal{B}_{H}(f, rn/h))m.
$
Using the definition of 
$\tau(\cdot)$ (Definition \ref{def:alexander}) we have
\[
1 + 2eP_X(\dis(\mathcal{B}_{H}(f, rn/h)))m
\le 1 + 2e \tau\left(\frac{r}{h}\right) \frac{r}{h} m
\le 11e\tau\left(\frac{r}{h}\right)\frac{\log(|\mathcal{M}_{r}|)}{h}.
\]
With probability at least $\frac{1}{2}$,
\[
|\{X_{1}, \ldots, X_{m}\} \cap \dis(\mathcal{B}_{H}(f, rn/h))| \le 11e\tau\left(\frac{r}{h}\right)\frac{\log(|\mathcal{M}_{r}|)}{h}.
\]
Using the union bound, we have that with probability greater than zero there exists a sequence of at most $11e\tau\left(\frac{r}{h}\right)\frac{\log(|\mathcal{M}_r|)}{h}$ elements, such that all functions in $\mathcal{M}_r$ classify this sequence distinctly. By the Vapnik and Chervonenkis lemma, we therefore have that
\[
|\mathcal{M}_{r}| \le \left(\frac{11e^2\tau\left(\frac{r}{h}\right)\frac{\log(|\mathcal{M}_{r}|)}{h}}{d}\right)^{d}.
\]
Using Corollary $4.1$ from \cite{Vid03} we have
\[
\log(|\mathcal{M}_{r}|) \le 2d\log\left(11e^2\tau\left(\frac{r}{h}\right)\frac{1}{h}\right).
\]
Using $\tau\left(\frac{r}{h}\right) \le  \mathbf{s} \wedge \frac{h}{r} \leq \mathbf{s} \wedge \frac{nh}{\gamma}$ (Theorem $10$ in \cite{Hanneke15}) we finally have
\[
\log(\mathcal{M}^{\text{loc}}_{1}(\F, \gamma, n, h)) \le 2d\log\left(11e^2\left(\frac{n}{\gamma} \wedge \frac{\mathbf{s}}{h}\right)\right).
\]
Now we upper bound $\gamma^{\text{loc}}_{h, h}(n)$, knowing that
\[
h\gamma^{\text{loc}}_{h, h}(n) \le 2d\log\left(11e^2\left(\frac{n}{\gamma^{\text{loc}}_{h, h}(n)} \wedge \frac{\mathbf{s}}{h}\right)\right).
\]
We obviously have $\gamma^{\text{loc}}_{h, h}(n) \le \frac{2d\log\left(11e^2\frac{\mathbf{s}}{h}\right)}{h}$. For $\gamma = \frac{2d\log\left(11e^2\frac{nh}{d}\right)}{h}$ we have $h\gamma = 2d\log\left(11e^2\frac{nh}{d}\right)$, but $2d\log\left(11e^2\frac{n}{\gamma}\right) \le 2d\log\left(11e^2\frac{nh}{d}\right)$ if $h > \frac{d}{11en}$. Finally,  we have 
\[
\gamma^{\text{loc}}_{h, h}(n) \le \frac{2d\log\left(11e^2\left(\frac{nh}{d} \wedge \frac{\mathbf{s}}{h}\right)\right)}{h}.
\]
Now we prove the lower bound. 
From \eqref{expectbound} established above, we know that $\frac{\gamma^{\text{loc}}_{h, h}(n)}{n}$ is, up to an absolute constant, a distribution-free upper bound for $\E(R(\hat{f}) - R(f^{*}))$, holding for all ERM learners $\hat{f}$. Then any lower bound on $\sup\limits_{P \in \mathcal{P}(h,\F)} \E(R(\hat{f}) - R(f^{*}))$ holding for any ERM learner is also a lower bound for $\frac{\gamma^{\text{loc}}_{h, h}(n)}{n}$. In particular, it is known \cite{Massart06,Hanneke15a} that for any learning procedure $\tilde{f}$, if $h \geq \sqrt{\frac{d}{n}}$, then $\sup\limits_{P \in \mathcal{P}(h,\F)} \E(R(\tilde{f}) - R(f^{*})) \gtrsim \frac{d + (1-h)\log(n h^2 \wedge \mathbf{s})}{nh}$, while if $h < \sqrt{\frac{d}{n}}$ then $\sup\limits_{P \in \mathcal{P}(h,\F)} \E(R(\tilde{f}) - R(f^{*})) \gtrsim \sqrt{\frac{d}{n}}$.  Furthermore, in the particular case of ERM, \cite{Hanneke15a} proves that any upper bound on $\sup\limits_{P \in \mathcal{P}(1,\F)} \E(R(\hat{f}) - R(f^*))$ holding for all ERM learners $\hat{f}$ must have size, up to an absolute constant, at least $\frac{\log(n \wedge \mathbf{s})}{n}$.  
Together, these lower bounds imply $\gamma^{\text{loc}}_{h, h}(n) \gtrsim \frac{d + \log(n h^2 \wedge \mathbf{s})}{h} \wedge \sqrt{d n}$.
\end{proof}
\section{Minimax Lower Bound}
In this section we prove that under Massart's bounded noise condition, fixed points of the local empirical entropy appear in minimax lower bounds. Results are in expectation and generally use classic lower bound techniques from the literature \cite{Massart06,Raginsky11,Yang99}, previously used only for specific classes. We will need the following definition, which will be motivated below.
\begin{definition}
\label{pseudoconvex}
Fix a class of classifiers $\F$. Assume that there exists a positive constant $c \ge 1$ such that for any $N$ in the definition of $\mathcal{M}^{\text{\emph{loc}}}_{1}(\F, \gamma^{\text{\emph{loc}}}_{h, 1}(N), N, 1)$ the supremum with respect to the radius is achieved at some $\varepsilon_{h}(N) \le c \gamma^{\text{\emph{loc}}}_{h, 1}(N)$. This class will be referred to as $c$-\emph{pseudoconvex}.
\end{definition}
\begin{theorem}
\label{newlowerbound}
Let $\tilde{f}$ be the output of any learning algorithm. Fix any $c_{\F}$-pseudoconvex class $\F$ and any $h$ satisfying $\sqrt{\frac{d}{n}} \le h \le 1$. Then there exists a $P \in \mathcal{P}(h, \F)$ such that
\begin{equation}
\label{lowerbound}
\E(R(\tilde{f}) - R(f^{*})) \gtrsim \frac{d}{nh} + \frac{1}{c_{\F}}\frac{(1 - h)\gamma^{\text{\emph{loc}}}_{h, 1}\left(\lceil\frac{nc_{\F}h}{(1 - h)}\rceil \right)}{n}.
\end{equation}
\end{theorem}
Conditions involving the constant $c_{\F}$ can be relaxed in different ways. It will be clear from our proof that we may remove the pseudoconvexity assumptions by redefining the local empirical entropy \eqref{localentropy} by removing the supremum with respect to the radius. Alternatively one can remove the supremum by introducing certain monotonicity assumptions. We note that related monotonicity assumptions were used implicitly in previous papers \cite{Gine06, Raginsky11}. In both relaxations our lower bound holds with $c_{\F} = 1$. Moreover, the bound \eqref{lowerbound} is valid for an arbitrary class $\F$ as we may always consider $c_{\F}(N)$ instead of $c_{\F}$, which is a minimal natural number satisfying $\varepsilon_{h}(N) \le c_{\F}(N) \gamma^{\text{loc}}_{h, 1}(N)$. Finally, we note that these monotonicity problems do not appear for convex classes, as noted by Mendelson in \cite{Mendelson15}. This is our motivation for the name of the condition in Definition \ref{pseudoconvex}: local entropy of the class has almost the same monotonicity properties as in the convex case. In the next section we will present examples of natural pseudoconvex classes.

The next lemma is given in \cite{Massart03} (Corollary $2.18$).
\begin{lemma}[Birg\'{e}]
\label{birgeslemma}
Let $\{P_{i}\}_{i = 0}^{N}$ be a finite family of distributions defined on the same measurable space and $\{A_{i}\}_{i = 0}^{N}$ be a family of disjoint events. Then
\[
\min\limits_{0 \le i \le N}P_{i}(A_{i}) \le 0.71 \vee \frac{\sum\limits_{i = 1}^{N}{\text{KL}}(P_{i} \| P_{0})}{N\log(N + 1)}.
\]
\end{lemma}
\begin{proof}[Theorem~\ref{newlowerbound}]
First we consider the value $\mathcal{M}^{\text{loc}}_{1}(\F, \gamma^{\text{loc}}_{h, 1}(N), N, 1)$. Recall that the definition of this value considers suprema over $f \in \F$ and over $N$-element subsets of $\mathcal{X}$. Without loss of generality we assume that these suprema are achieved at some classifier $g \in \F$, some $\varepsilon_{h}(N) \in [\gamma^{\text{loc}}_{h, 1}(N), N]$ and at some particular set $\mathcal{X}_{N} = \{x_{1}, \ldots, x_{N}\}$. Let $k_i$ define the number of copies of $x_i$ in $\mathcal{X}_{N}$. We define $P_{\mathcal{X}_{N}}(\{x_{i}\}) = \frac{k_i}{N}$. If all elements are distinct this measure is just a uniform measure on $\mathcal{X}_{N}$.  We introduce a natural parametrization: any classifier is represented by an $N$-dimensional binary vector and two vectors (for classifiers $g,f$) disagree only on a set corresponding to $\dis(\{g, f\}) \cap \mathcal{X}_{N}$. The set of binary vectors corresponding to classifiers in $\F$ will be denoted by $\mathcal{B}$. For a given binary vector $b$ define $P_b = P_\mathcal{X_{N}}\times P^b_{Y|X}$, where $P^b_{Y = 1|X_{i}} = \frac{1 + (2b_i - 1)h}{2}$.  Let $\tilde{f}_{b}$ denote the classifier $\tilde{f}$ produced by the learning algorithm when $P_b$ is the data distribution, and let $\tilde{b}$ denote the binary vector corresponding to $\tilde{f}_{b}$; thus, $\tilde{b}$ is a random vector, which depends on the parameter $b$ only through the $n$ data points having distribution $P_b$.
It is known \cite{Devr95} that $R(\tilde{f}) - R(f^{*}) = \E(|\eta(X)|\Ind[\tilde{f}(X) \neq f^{*}(X)] | \tilde{f}) \ge hP((x,y) : \tilde{f}(x) \neq f^{*}(x))$, when $P \in \mathcal{P}(h, \F)$.  
Furthermore, when $P_{b}$ is the data distribution, we have $P_{b}((x,y) : \tilde{f}_{b}(x) \neq f^{*}(x)) = \frac{\rho_{H}(\tilde{b},b)}{N}$.
Thus, we have 
\begin{align*}
&\sup\limits_{P \in \mathcal{P}(h, \F)}\E(R(\tilde{f}) - R(f^{*})) 
\ge \max\limits_{b \in \mathcal{B}} \E\left( h P_{b}((x,y) : \tilde{f}_{b}(x) \neq f^{*}(x)) \right)
\\
&\ge \frac{h}{N} \max\limits_{b \in \mathcal{B}} \E(\rho_{H}(\tilde{b}, b)).
\end{align*}

Let $b^{*}$ be the binary vector in $\mathcal{B}$ corresponding to the classifier $g$ defined above, and fix a maximal subset $\mathcal{B}^{\text{loc}} \subset \mathcal{B}$ satisfying the properties that for any $b' \in \mathcal{B}^{\text{loc}}$ we have $\rho_{H}(b', b^*) \le \varepsilon_{h}(N)$ and for any two $b^{\prime},b^{\prime\prime} \in \mathcal{B}^{\text{loc}}$ we have $\rho_{H}(b^{\prime}, b^{\prime\prime}) > \varepsilon_{h}(N)/2$.
Next, define $\breve{b}$ as the minimizer of $\rho_{H}(\breve{b},\tilde{b})$ among $\mathcal{B}^{\text{loc}}$.  In particular, if $b \in \mathcal{B}^{\text{loc}}$, we have $\rho_{H}(\breve{b},\tilde{b}) \leq \rho_{H}(b,\tilde{b})$, so that $\rho_{H}(\breve{b},b) \leq \rho_{H}(\breve{b},\tilde{b}) + \rho_{H}(\tilde{b},b) \leq 2 \rho_{H}(\tilde{b},b)$.  Therefore,
\begin{equation*}
\frac{h}{N} \max\limits_{b \in \mathcal{B}} \E(\rho_{H}(\tilde{b},b))
\geq \frac{h}{N} \max\limits_{b \in \mathcal{B}^{\text{loc}}} \E(\rho_{H}(\tilde{b},b))
\geq \frac{h}{2N} \max\limits_{b \in \mathcal{B}^{\text{loc}}} \E(\rho_{H}(\breve{b},b)).
\end{equation*}

Recalling that $\breve{b}$ is a deterministic function of $\tilde{f}$, which itself is a function of the $n$ data points, we may define disjoint subsets $A_{b}$ of $(\mathcal{X} \times \mathcal{Y})^{n}$, for $b \in \mathcal{B}^{\text{loc}}$, where $A_{b}$ corresponds to the collection of data sets that would yield $\breve{b} = b$.\footnote{For simplicity, we are supposing the learning algorithm is not randomized; the argument easily extends to randomized algorithms by conditioning on the internal randomness in this step.}
Now, from Markov's inequality and the fact that the vectors in $\mathcal{B}^{\text{loc}}$ are $\frac{\varepsilon_{h}(N)}{2}$-separated, we have $\E(\rho_{H}(\breve{b}, b)) \ge \frac{\varepsilon_{h}(N)}{2}\P(\breve{b} \neq b) = \frac{\varepsilon_{h}(N)}{2}(1-P_{b}^{n}(A_{b}))$.  Thus we have that
\begin{equation*}
\frac{h}{2N} \max\limits_{b \in \mathcal{B}^{\text{loc}}} \E(\rho_{H}(\breve{b},b))
\geq \frac{h \varepsilon_{h}(N)}{4N} \left( 1 - \min\limits_{b \in \mathcal{B}^{\text{loc}}} P_{b}^{n}(A_{b}) \right).
\end{equation*}
We are interested in using Lemma~\ref{birgeslemma} to upper-bound $\min\limits_{b \in \mathcal{B}^{\text{loc}}} P_{b}^{n}(A_{b})$.  Toward this end, note that for any $b^{\prime},b^{\prime\prime} \in \mathcal{B}^{\text{loc}}$, standard calculations show that
\begin{equation*}
{\text{KL}}(P_{b^{\prime}}^n \| P_{b^{\prime\prime}}^n) = \frac{n}{N}h\ln\left(\frac{1 + h}{1 - h}\right)\rho_h(b^{\prime}, b^{\prime\prime}).
\end{equation*}
Because for $x > 0$ we have $\ln(x + 1) \le x$, it holds that $h\ln\left(\frac{1 + h}{1 - h}\right) \le \frac{2h^2}{1 - h}$. Furthermore, for any $b^{\prime}, b^{\prime\prime} \in \mathcal{B}^{\text{loc}}$ we have $\rho_{H}(b^{\prime}, b^{\prime\prime}) \le 2\varepsilon_{h}(N)$.  Therefore, 
\begin{equation*}
{\text{KL}}(P_{b^{\prime}}^n \| P_{b^{\prime\prime}}^n) \le \frac{4nh^2\varepsilon_{h}(N)}{N(1 - h)}.
\end{equation*}
Thus, by Lemma \ref{birgeslemma}, 
\begin{equation}
\label{min-pb-bound}
\min\limits_{b \in \mathcal{B}^{\text{loc}}} P_{b}^{n}(A_{b})
\leq 0.71 \lor \frac{\frac{4nh^2\varepsilon_{h}(N)}{N(1 - h)}}{\log(|\mathcal{B}^{\text{loc}}|)}.
\end{equation}
Noting that $\log(|\mathcal{B}^{\text{loc}}|) = \log(\mathcal{M}^{\text{loc}}_{1}(\F,\varepsilon_{h}(N),N,1)) \geq h \gamma^{\text{loc}}_{h,1}(N) \geq h \varepsilon_{h}(N) / c_{\F}$, choosing $N = \left\lceil \frac{6nc_{\F}h}{(1 - h)} \right\rceil$ yields 
\begin{equation*}
\frac{4nh^2\varepsilon_{h}(N)}{N(1 - h)} \leq \frac{2h\varepsilon_{h}(N)}{3c_{\F}} \leq \frac{2}{3} \log(|\mathcal{B}^{\text{loc}}|),
\end{equation*}
so that the right hand side of \eqref{min-pb-bound} is $0.71$.
Altogether, we have that for $h < 1$, 
\begin{align*}
&\sup\limits_{P \in \mathcal{P}(h, \F)}\E(R(\tilde{f}) - R(f^{*})) 
\geq 0.29 \frac{h \varepsilon_{h}(N)}{4N}
\\
&\geq \frac{0.29}{48 c_{\F}} \frac{(1-h)\varepsilon_{h}(N)}{n}
\geq \frac{0.29}{48 c_{\F}} \frac{(1-h)\gamma^{\text{loc}}_{h,1}(N)}{n}.
\end{align*}
The term $\frac{d}{nh}$ for $h > \sqrt{\frac{d}{n}}$ is a part of the classic lower bound of \cite{Massart06}. 
\end{proof}
The following observation is an important consequence of our analysis.
\begin{corollary}
\label{coroloptimal}
Consider a $c_{\F}$-pseudoconvex class $\F$. Let $0 < C_{0} \le C_{1} < 1$. Then if the margin parameter $h$ is such that $C_{0} \vee \sqrt{\frac{d}{n}} \le h \le  C_{1}$, then for any VC class $\F$ the ERM upper bound \eqref{expectbound} and the lower bound \eqref{lowerbound} match up to the constant factors (also appearing possibly in the argument of the fixed point), which may depend only on $C_{0}$, $C_{1}$ and $c_{\F}$.
\end{corollary}
It is known \cite{Massart06} that ERM is minimax optimal up to constant factors if $0 \le h < \sqrt{\frac{d}{n}}$. Interestingly, our corollary is certainly not valid for $C_{1} = 1$. The optimal bound in the realizable case is of order $\frac{d}{n} + \frac{\log(\frac{1}{\delta})}{n}$ \cite{Hanneke16,vapnik:74}, but ERM cannot generally have this convergence rate in the realizable case \cite{Haussler94,Auer07,Simon15}. In this special case, our lower bound recovers the classic 
$\frac{d}{n}$ lower bound of \cite{vapnik:74,Ehren89}, since when $h$ is close to $1$ 
the term $\frac{(1 - h)\gamma^{\text{loc}}_{h, 1}}{n c_{\F}}$ disappears and we have 
only the $\frac{d}{n}$ term.
\section{Estimation of a Fixed Point of Local Empirical Entropy for Specific Classes}
\label{examples}
In this section we provide two examples of exact estimation of fixed points of local empirical entropies. First we consider threshold classifiers, introduced in Example~\ref{threshold}. For this particular class, $d = 1$ and $\mathbf{s} = 2$. From Theorem \ref{mainupperbound} we have $\frac{1}{h} \lesssim \gamma_{h, h}^{\text{loc}}(n) \lesssim \frac{\log(\frac{1}{h})}{h}$, and explicit calculation for this special class reveals $\gamma_{h,h}^{\text{loc}}(n) \simeq \frac{\log(\frac{1}{h})}{h}$.  In particular, in the realizable case $\gamma_{1, 1}^{\text{loc}}(n) \simeq 1$. 

Another example will be a class of linear separators in $\mathbb{R}^{k}$ for $k \ge 2$. This class is known to have VC dimension $d = k + 1$. It is easy to verify that for this particular class $\mathbf{s} = \infty$ \cite{Hanneke15}.
\begin{proposition}
\label{linearsep}
For the set $\F$ of linear separators in $\mathbb{R}^{d}$, if $d \ge 2$, then for any $h > \sqrt{\frac{d}{n}}$
\[
\frac{d\log\left(\frac{nh^2}{d}\right)}{h} \lesssim  \gamma_{h, h}^{\text{\emph{loc}}}(n) \lesssim \frac{d\log\left(\frac{nh}{d}\right)}{h}.
\]
In particular, $\gamma_{1, 1}^{\text{\emph{loc}}}(n) \simeq d\log(\frac{n}{d})$.
\end{proposition}
\begin{proof}
The upper bound follows directly from the Theorem \ref{mainupperbound}.
At first we select a special set of points $x_{1}, \ldots, x_{n} \in \mathbb{R}^{d}$. It is known (Theorem $6.5$ in \cite{Edels87}) that in $\mathbb{R}^{d}$ there exists a so called \emph{cyclic polytope} with $n$ vertices, such that it has exactly $n \choose k$\  $(k - 1)$-dimensional faces for any $k \le \lfloor\frac{d}{2}\rfloor$. We choose $x_{1}, \ldots, x_{n}$, such that $x_{i}$ is a vertex of the cyclic polytope. We fix any linear separator $f_{1}$ such that all $x_{i}, \ldots ,x_{n}$ are in the same half-space with respect to this linear separator. Without loss of generality we may assume that $f_{1}(x_{1}) = \ldots = f_{1}(x_{n}) = -1$. In this notation using the property of cyclic polytopes we see that $\F$ contains all classifiers with at most $\lfloor\frac{d}{2}\rfloor$ ones. We denote this set by $\F_{d/2}$. Analysis of this particular set by Massart and N\'ed\'elec (Theorem $5$ in \cite{Massart06}) gives a $\frac{(1 - h)d\log(\frac{nh^2}{d})}{nh}$ lower bound for $R(\hat{f}) - R(f^*)$ provided that $h > \sqrt{\frac{d}{n}}$. From Theorem \ref{mainupperbound} we know that this lower bound is also a lower bound for $\gamma_{h, h}^{\text{loc}}(n)$. Thus $\frac{(1 - h)d\log(\frac{nh^2}{d})}{h} \lesssim \gamma_{h, h}^{\text{loc}}(n)$. Simultaneously, we have $\gamma_{1, 1}^{\text{loc}}(n) \le \gamma_{h, h}^{\text{loc}}(n)$. So, it is enough to lower bound $\gamma_{1, 1}^{\text{loc}}(n)$, which may be derived as a lower bound for ERM in the realizable case. It is known (theorem $6$ in \cite{Simon15}, or theorem $5$ in \cite{Auer07}) that for this particular class $\F_{d/2}$ in the realizable case there exists ERM such that with probability at least $\frac{1}{2}$ we have $\frac{d\log(\frac{n}{d})}{n} \lesssim R(\hat{f})$. This implies that $\frac{d\log(\frac{n}{d})}{n} \lesssim \E R(\hat{f})$ and thus $d\log(\frac{n}{d}) \lesssim \gamma_{1, 1}^{\text{loc}}(n)$. Summerizing, we have $d\log(\frac{n}{d}) \vee \frac{(1 - h)d\log(\frac{nh^2}{d})}{h} \lesssim \gamma_{h, h}^{\text{loc}}(n)$. We finish the proof by noticing that $\frac{d\log\left(\frac{nh^2}{d}\right)}{h} \lesssim d\log(\frac{n}{d}) \vee \frac{(1 - h)d\log(\frac{nh^2}{d})}{h}$.
\end{proof}
We note that the lower bound \eqref{lowerbound} may be applied for both classes. 
\section{Discussion and Open Problems}
Local entropies are well known in statistics since the early work of Le Cam \cite{Lecam73}. Since then local metric entropies have appeared in minimax lower bounds \cite{Yang99,Mendelson15,Lecue13} and in the necessary and sufficient conditions for consistency of ERM estimator in nonparametric regression \cite{Geer96}. Simultaneously, the upper bounds are usually given in terms of global entropies. Interestingly, it is sometimes possible to recover optimal rates by considering only global packings \cite{Yang99,Rakhlin13}. Generally, empirical covering numbers of classes in statistics have two types of behaviour. There are \emph{parametric} and \emph{VC-type} classes where the logarithm of covering numbers scales as $\log(\frac{1}{\varepsilon})$ and expressive \emph{nonparametric classes} where it scales as $\varepsilon^{-p}$ for some $p > 0$. It was proven in \cite{Yang99} that for these expressive nonparametric classes local and global entropies are of the same order. Thus for such classes localization of class does not give any significant improvement and minimax rates are usually obtained using only global entropies \cite{Rakhlin13}. The case of parametric and especially VC-type classes is more delicate and this paper is a first attempt to analyze the last tightly under bounded noise \footnote{We note that for some parametric classes, specifically for a bounded subset of finite dimensional linear space in $L_{2}$, optimal rates were obtained in \cite{Koltch11}}. Our results and examples show that localization of the class is usually needed for VC classes, but definitely not always. Some parametric classes have the features of nonparametric classes: their local entropies are of the same order as their global entropies, and for them bounds in terms of global entropies are essentially optimal. It is not difficult to show that, in the proof of Proposition \ref{linearsep}, we gave an example of such a VC class $\F_{d/2}$. Not surprisingly, Massart and N\'ed\'elec \cite{Massart06} named this class \emph{rich}. This class appears in almost all class-specific lower bounds \cite{Raginsky11, Massart06, Rakhlin13}, which are matched by global upper bounds. In contrast, there are still many interesting classes, for example, threshold classifiers, which are out of the scope of upper bounds based on global entropy.

We should note that a distribution-dependent local entropy has already appeared in the upper bounds in the classification literature under the name of the \emph{doubling dimension}. Given a class of classifiers $\F$ and a probability distribution $P_{X}$, define the doubling dimension by 
\begin{equation}
\label{doublingdimension}
D(\F, \gamma) = \max\limits_{f \in \F}\max\limits_{\varepsilon \ge \gamma}\log(\mathcal{N}(\mathcal{B}_{P_{X}}(f, \varepsilon), \varepsilon/2)),
\end{equation}
where $\mathcal{B}_{P_{X}}(f, \varepsilon) = \{g \in \F| P_{X}(f(X) \neq g(X)) \le \varepsilon\}$ and $\mathcal{N}(\G, \varepsilon)$ is the $\varepsilon$-covering number of $\G$ with respect to the pseudo-metric $P_{X}(g(X) \neq g'(X))$. It was proved by Bshouty, Li, and Long \cite{Bshouty09} that in the realizable case, for any $\varepsilon > 0$, if
\[
n \gtrsim \frac{d + D(\F, \varepsilon_{0})}{\varepsilon}\sqrt{\log\left(\frac{1}{\varepsilon}\right)} + \frac{\log(\frac{1}{\delta})}{\varepsilon},
\]
then with probability at least $1 - \delta$, for any ERM $\hat{f}$ we have $R(\hat{f}) \le \varepsilon$. Here $\varepsilon_{0} = \varepsilon\exp\left(-\sqrt{\log(\frac{1}{\varepsilon})}\right)$.
It is easy to show that when considering the distribution-free setting, this bound is weaker than ours at least because it contains a square root of an extra logarithmic factor. The following simple inequality compares distribution-free doubling dimension and the local empirical entropy.
For any $\gamma \in \mathbb{N}$,
\begin{equation}
\label{eqn:local-empirical-vs-doubling}
\log(\mathcal{M}^{\text{loc}}_{1}(\F, \gamma, n, 1)) \le 2\sup\limits_{P_{X}}D(\F, \gamma/n).
\end{equation}
To prove this inequality one may consider the uniform probability measure $P_{X}$ on the $n$ points maximizing the local packing number on the left hand side, in which case the pseudo-metric $P_{X}(g(X) \neq g'(X))$ is merely $1/n$ times the Hamming distance of the projections to these $n$ points.  The constant $2$ appears simply due to the fact that empirical local entropies involve packing numbers while the doubling dimension involves covering numbers.  Bshouty, Li, and Long \cite{Bshouty09} also study a non-ERM distribution-dependent learning algorithm in the realizable case, and obtain an error rate guarantee essentially bounded by a fixed point $\varepsilon \approx \frac{D(\F,\varepsilon/4)}{n} + \frac{\log(\frac{1}{\delta})}{n}$, with probability at least $1-\delta$.  In light of  \eqref{eqn:local-empirical-vs-doubling}, we see that in the worst case over distributions this is essentially no better than our Theorem~\ref{mainupperbound} (with $h=1$), which holds for the much-simpler learning algorithm ERM.

We note that questions similar to ours have been considered recently by Mendelson~\cite{Mendelson15} and by Lecu\'e and Mendelson \cite{Lecue13}. Both papers introduce distribution dependent fixed points of local entropies and show that in the convex regression setup for subgaussian classes they give optimal upper and lower bounds. However, the direct comparison with their results is problematic due to the fact that in the VC case we do not have convexity assumptions: they are replaced by noise assumptions and specifically used by our approach. Moreover, since in the realizable case ERM is not minimax optimal, it can be easily seen from our results that there may not exist a lower bound in terms of fixed points of the local empirical entropy in this case.

We have compared our bound with some of the best known relaxations of the bounds based on local Rademacher processes \eqref{koltchbound}. However, the title of our paper demands also a direct comparison with the bounds based \emph{solely} on local Rademacher complexities. For this, we need the following result.\begin{theorem}[Sudakov minoration for Bernoulli process \cite{Talagrand14}]
Let $V \subset \mathbb{R}^n$ be a finite set such that for any $v_{1}, v_{2} \in V$ if $v_{1} \neq v_{2}$ then $\|v_1 - v_2\|_{2} \ge a$ for some $a > 0$ and for any $v \in V$ it holds $\|v\|_{\infty} \le b$ for some $b > 0$. Then
\begin{equation}
\label{minoration}
\E_{\varepsilon} \sup\limits_{v \in V}\sum\limits_{i = 1}^{n}\varepsilon_{i}v_{i} \gtrsim a\sqrt{\log{|V|}} \wedge \frac{a^2}{b}.
\end{equation}
\end{theorem}
For simplicity we will consider only the realizable case, and distribution-free setting. However we note that similar arguments will also work under bounded noise and general distributions $P_{X}$. Fix a sample $x_{1}, \ldots, x_{n}$. Applying Corollary $5.1$ from \cite{Bartlett05} we have
\[
\E R(\hat{f}) \lesssim \sup\limits_{x_{1}, \ldots, x_{n}}r^{*}, 
\]
where $r^{*}$ is a fixed point of the local empirical Rademacher complexity, that is a solution of the following equality
\[
\frac{1}{n}\E_{\varepsilon} \sup\limits_{g \in \mstar(\G_{f^{*}}), P_{n}g \le 2r}\sum\limits_{i = 1}^{n}\varepsilon_{i}g(x_{i}) = r, 
\]
where $\mstar(\G)$ denotes the \emph{star-hull} of a class $\G$: that is, the class of functions $\alpha g$, where $g \in \G$ and $\alpha \in [0, 1]$. Since $\mstar(\G_{f^{*}})$ is star-shaped, it can be simply proven (see appropriate discussions in \cite{Mendelson15}) that local empirical entropies are not increasing in its radius. Using this fact together with \eqref{minoration} it can be shown 
\[
\E_{\varepsilon} \sup\limits_{g \in \mstar(\G_{f^{*}}), P_{n}g \le \frac{2\gamma}{n}}\sum\limits_{i = 1}^{n}\varepsilon_{i}g(x_{i}) \gtrsim \sqrt{\gamma}\sqrt{\log(\mathcal{M}^{\text{loc}}_{1}(\F, \gamma, n, 1))} \wedge \gamma.
\] From this it easily follows that $\frac{\gamma_{1, 1}^{\text{loc}}(n)}{n} \lesssim r^*$. Thus our bounds are not generally worse than the bounds based \emph{solely} on local Rademacher complexities. Conceptually we are looking for fixed points of the right hand side of \eqref{minoration}, while Rademacher analysis works directly with the fixed points of the suprema of localized processes. 

There are still interesting questions and possible directions that are out of the scope of this paper:
\begin{enumerate}
\item We are focusing on a distribution-free analysis. At the same time by just leaving the expectations with respect to the learning sample we may simply obtain a distribution-dependent version of Theorem \ref{mainupperbound}. Recently, Balcan and Long \cite{Balcan13} have proven that for some special distributions $P_{X}$, the class of homogenous linear separators admits faster rates of convergence of ERM, compared to worst-case distributions. It may be interesting to generalize our results using distribution-dependent fixed points of the local empirical entropy (based on random data, rather than worst-case data), and specifically to determine whether this yields rates as fast as \cite{Balcan13} under similar conditions on $P_{X}$.
\item Our approach here makes use of shifted processes and offset Rademacher processes, in place of explicit diameter-localization arguments such as used by \cite{Koltch06}.  It seems a natural direction to develop a more general theory of this use to understand the limitations of the approach. For example,  so far our analysis is specific to the well-specified case when $f^{*} \in \F$. It would be interesting to generalize our results to more general noise conditions and a miss-specified case.
\item It is also interesting to refine our bounds in situations when $h$ is close to zero: i.e., when the noise levels are high. It is known \cite{Massart06} that when $h < \sqrt{\frac{d}{n}}$ the control of Rademacher processes based on the \emph{Dudley integral} \cite{Dudley84} give minimax optimal $\sqrt{\frac{d}{n}}$ convergence rate. Moreover, it is known that bounds based on just one covering are suboptimal in this case. If we fix $h = \sqrt{\frac{d}{n}}$, then the bound of Gin\'e and Koltchinskii \eqref{koltchbound} (also based on the Dudley integral) will give us an optimal $\sqrt{\frac{d}{n}}$ rate in expectation. Simultaneously, we know that their bound is suboptimal when $h$ is close to $1$. Due to an extra term $\frac{d\log\left(\frac{1}{h}\right)}{h}$ in \eqref{controloffixedpoint} our bound \eqref{expectbound} can guarantee only a suboptimal $\sqrt{\frac{d}{n}}\log\left(\frac{n}{d}\right)$ rate when $h = \sqrt{\frac{d}{n}}$, but we know that for many other values of $h$ our bound is significantly better. Nonetheless, we believe that there is a transition, continuous in $h$, from the Dudley integral regime when $h < \sqrt{\frac{d}{n}}$ to the regime when the local empirical entropy provides the optimal characterization of the rates obtained by ERM.
\item We have already discussed that ERM may be suboptimal in the realizable case. Thus, when considering minimax optimality there is a third regime, when we have almost no noise. However, since ERM is such a natural and frequently-used method, it remains an interesting question to precisely characterize its risk.
Recall that the case when $h$ is bounded away from $0$ and $1$ is partially covered by our Corollary \ref{coroloptimal}. We hypothesize that in the realizable case (and even in a more general regime when $h$ is close to $1$) our bound \eqref{devbound} also characterizes the \emph{best possible} bound on the risk of the worst-case choice of empirical risk minimizer $\hat{f}$, up to an absolute constant factor. It follows directly from our discussions that our hypothesis is true for the classes presented in Section \ref{examples}. Partial analysis of the complexity of ERM has recently been performed by Hanneke \cite{Hanneke15a}. Specifically, he finds that the correct characterization of the risk of ERM is somewhere between the upper bounds \eqref{Ginecorollary}, \eqref{s-over-d-bound} and a lower bound
\begin{equation}
\label{loose-lb}
R(\hat{f}) - R(f^{*}) \gtrsim \frac{d}{nh} + \frac{\log(nh^2 \wedge \mathbf{s})}{nh} + \frac{\log(\frac{1}{\delta})}{nh},
\end{equation}
holding with probability greater than $1 - \delta$ for a worst-case choice of $P \in \mathcal{P}(h,\F)$ (and worst-case choice of ERM).  We know that in the realizable case, for the class presented in Example \ref{twoclassesexample}, the bound \eqref{loose-lb} is matched. At the same time, for the class of linear separators presented in Section \ref{examples}, this lower bound is not tight. This, in particular, leads to the obvious conclusion that $d$ and $\mathbf{s}$ are also not sufficient to fully characterize the risk of ERM, even in the realizable case.
\end{enumerate}
\section*{Acknowledgement}
The authors would like to thank Sasha Rakhlin for his suggestion to use offset Rademacher processes to analyze binary classification under Tsybakov noise conditions and anonymous reviewers of the short version of this paper for their helpful comments. NZ was supported solely by the Russian Science Foundation grant (project 14-50-00150).

\bibliographystyle{authordate1}

\appendix
\section{Proofs}

We now have the proof of Theorem~\ref{startheorem}.

\begin{proof}[Theorem~\ref{startheorem}]
Let $\dis_{0}$ be a disagreement set of the version space of first $\lfloor n/2 \rfloor$ instances of the learning sample. The random error set will be denoted by $E_{1} = \{x \in \mathcal{X}|\hat{f}(x) \neq f^{*}(x)\}$. 
We continue the notational conventions from the proof of Proposition \ref{coveringbound}.
Also recall that $R_{n}(\hat{f}) = 0$ in this context. 
Using symmetrization Lemma \ref{symmetrization} and Lemma \ref{expectmax} we have for any $c > 0$
\[
\E P(E_{1}) = \E R(\hat{f}) \le \E\sup\limits_{g \in \G_{f^*}}(Pg - (1 + c)P_{n}g) \le \frac{2\left(1 + \frac{c}{2}\right)^2}{c}\frac{\log\left(\mathcal{S}_{\mathcal{F}}\left(n\right)\right)}{n}.
\]
We fix $c = 2$ and prove that for any distribution $\E P(E_{1}) \le \frac{4\log\left(\mathcal{S}_{\mathcal{F}}\left(n\right)\right)}{n}$.
Observe that $E_{1} \subseteq \dis_{0}$. Now we may use that
$
R(\hat{f}) = P(E_{1}|\dis_{0})P(\dis_{0}).
$
Let $\xi = |\dis_{0} \cap \{X_{\lfloor n/2 \rfloor + 1}, \ldots, X_{n}\}|$. Conditionally on the first $\lfloor n/2 \rfloor$ instances $\xi$ has binomial distribution. Expectations with respect to the first and the last parts of the sample will be denoted respectively by $\E$ and $\E'$. Conditionally on $\{X_{1}, \ldots, X_{\lfloor n/2 \rfloor}\}$ we introduce two events
\begin{align*}
&A_{1}: \xi < \frac{nP(\dis_{0})}{4},
\\
&A_{2}: \xi > \frac{3nP(\dis_{0})}{4}.
\end{align*}
Using Chernoff bounds we have $P(A_{j}) \le \exp\left(-\frac{nP(\dis_{0})}{16}\right)$, $j = 1, 2$. Denote $A = A_{1}\cup A_{2}$. Then
\[
\E'P(E_{1}|\dis_{0}) =  \E'\left[P(E_{1}|\dis_{0})\Big|\overline{A}\right]P(\overline{A}) + \E'\left[P(E_{1}|\dis_{0})\Big|A\right]P(A).
\]
For the first term we have
\begin{equation*}
\E'\left[P(E_{1}|\dis_{0})\Big|\overline{A}\right]P(\overline{A})
\le \E'\left[P(E_{1}|\dis_{0})\Big|\overline{A}\right]
\le 
\frac{16\log\left(\mathcal{S}_{\mathcal{F}}\left(\left\lfloor \frac{3nP(\dis_{0})}{4} \right\rfloor\right)\right)}{nP(\dis_{0})}.
\end{equation*}
For the second term multiplied by $P(\dis_{0})$ we have
\begin{align*}
\E'\left[P(E_{1}|\dis_{0})\Big|A\right]P(\dis_{0})P(A) 
& \le 2\E'P(\dis_{0})\exp\left(-\frac{nP(\dis_{0})}{16}\right)
\\
= 2P(\dis_{0})\exp\left(-\frac{nP(\dis_{0})}{16}\right)
& \le
\frac{12}{n}.
\end{align*}
Combining the above we have
\[
\E'P(E_{1}|\dis_{0})P(\dis_{0}) \le \frac{16\log\left(\mathcal{S}_{\mathcal{F}}\left(\left\lfloor \frac{3nP(\dis_{0})}{4} \right\rfloor\right)\right)}{n}
 + \frac{12}{n}.
\]
It easy to see that for all $k, r \in \mathbb{N}$
\[
\left(\mathcal{S}_{\mathcal{F}}(kr)\right)^{\frac{1}{r}} \le \mathcal{S}_{\mathcal{F}}(k).
\]
Therefore, we have

\begin{align*}
\E R(\hat{f}) 
&\le \E\left( \frac{16\log\left(\mathcal{S}_{\mathcal{F}}\left(\left\lfloor \frac{3nP(\dis_{0})}{4} \right\rfloor \right)\right)}{n}  + \frac{12}{n} \right)
\\
&\le \E\frac{16\log\left(\mathcal{S}_{\mathcal{F}}\left(\mathbf{s} \max\left\{1, \left\lceil \frac{3nP(\dis_{0})}{4\mathbf{s}} \right\rceil \right\}\right)\right)}{n} + \frac{12}{n}
\\
&\le
\frac{16\E\max\left\{1, \left\lceil \frac{3nP(\dis_{0})}{4\mathbf{s}} \right\rceil \right\}\log\left(\mathcal{S}_{\mathcal{F}}\left(\mathbf{s}\right)\right)}{n}  + \frac{12}{n} 
\\
&\le
\frac{16\left(1 + \frac{3}{2}\right)\log\left(\mathcal{S}_{\mathcal{F}}\left(\mathbf{s}\right)\right)}{n}  + \frac{12}{n}
=
\frac{40\log\left(\mathcal{S}_{\mathcal{F}}\left(\mathbf{s}\right)\right)}{n}  + \frac{12}{n},
\end{align*}
where the fourth inequality uses \eqref{starlabel}.
The proof of the deviation bound is completely analogous, but slightly more technical. We refer to the proof of Theorem $11$ in \cite{Hanneke15a}, which can be easily generalized to our case.
\end{proof}

We now present the proof of the deviation bound in Proposition~\ref{coveringbound}.

\begin{proof}[Proposition \ref{coveringbound} Deviation Bound]
The proof in deviation is based on the symmetrization Lemma \ref{symmetrizationdev} and 
Corollary \ref{sym}. First we notice that $\mathcal{G}_{f^{*}}$ is a $(1, 1)$-Bernstein class. 
Thus fixing any $c_1, c_2$, such that $0 < c_2 < c_1$ and $t \ge \frac{(1 + c_{2})^2}{nc_{2}}$ 
it is sufficient to control 
$\Prob\left(\sup\limits_{g \in \G_{f^{*}}}\left(\frac{1 + c'/2}{n}\sum\limits_{i = 1}^{n}\varepsilon_{i}g_i - c'P_{n}g/2\right) \ge y\right)$ for some $y > 0$, where $g_i = g(Z_i)$.
Using the same decomposition as in Lemma \ref{expectmaxnew}  we have for a fixed 
$\lambda > 0$, $x > 0$ and $c'' = (1 + c'/2)(1 +\frac{c'}{4(1 + c'/2)})$
\begin{align*}
&\Prob\left(\sup\limits_{g \in \G_{f^{*}}}\left(\frac{1 + c'/2}{n}\sum\limits_{i = 1}^{n}\varepsilon_{i}g_i - c'P_{n}g/2\right) \ge x + \frac{\gamma c''}{n}\right) 
\\
&\le\! \Prob\left(\!\frac{\gamma c''}{n}\! +\! \sup\limits_{g \in \G_{f^{*}}}\!\left(\frac{1\! +\! c'/2}{n}\sum\limits_{i = 1}^{n}\varepsilon_{i}p(g)_i\! -\! c'P_{n}p(g)/2\right)\! \ge\! x\! +\! \frac{\gamma c''}{n}\right) 
\\
&\le \exp(-\lambda{x}n)\E\E_{\varepsilon}\exp\left(\lambda(1\!+\!c'/2)\sup\limits_{g \in \G_{f^{*}}}\!\!\left(\sum\limits_{i = 1}^{n}\varepsilon_{i}p(g)_i \!-\!\frac{c'}{c' + 2}p(g)_i\right)\right),
\end{align*}
where, as in Lemma \ref{expectmaxnew}, the operator $p$ denotes the nearest element in the $\gamma$-covering.
By denoting $c''' = \frac{c'}{c' + 2}$ and $\lambda' = \lambda(1 + c'/2)$ we have
\begin{align*}
&\E_{\varepsilon}\exp\left(\lambda'\sup\limits_{g \in \G_{f^{*}}}\left(\sum\limits_{i = 1}^{n}\varepsilon_{i}p(g)_i - c'''P_{n}p(g)\right)\right) 
\\
&\le \mathcal{M}^{*}_{1}(\F, \gamma, n)\exp\left(\sum\limits_{i = 1}^{n}\left(\frac{(\lambda')^2}{2}p(g)_i - \lambda'c'''p(g)_i\right)\right).
\end{align*}
Setting $\lambda = 2 c''' / (1 + c' / 2)$ so $\lambda' = 2c'''$ we have
\begin{align*}
&\Prob\left(\sup\limits_{g \in \G_{f^{*}}}\left(\frac{1 + c'/2}{n}\sum\limits_{i = 1}^{n}\varepsilon_{i}g_i - c'P_{n}g/2\right) \ge x + \frac{\gamma c''}{n}\right) 
\\
&\le \exp\left(-\frac{4c'xn}{(2 + c')^2}\right)\mathcal{M}^{*}_{1}(\F, \gamma, n).
\end{align*}
We set $x = \frac{(2 + c')^2}{4c'}\left(\frac{\log(\mathcal{M}^{*}_{1}(\F, \gamma, n))}{n} + \frac{\log(\frac{4}{\delta})}{n}\right)$ and choose $c_{1} = 3$ and $c_{2} = 1$. Then with probability at least $1 - \delta$,
\[
\sup\limits_{g \in \G_{f^{*}}}(P - (1 + c_1)P_{n})g \lesssim \frac{\gamma}{n} + \frac{\log(\mathcal{M}^{*}_{1}(\F, \gamma, n))}{n} + \frac{\log(\frac{1}{\delta})}{n}.
\]
We finish the proof by setting $\gamma = \gamma^{*}_{\frac{1}{2}}(n) + 1$. The upper bound \eqref{fixedpointupperbound} easily follows from the general bound on packing numbers for VC classes \cite{Haussler95}.
\end{proof}

Next, we have the proof of Lemma~\ref{localization}.
\begin{proof}[Lemma \ref{localization}]
Once again, given $X_1,\ldots,X_n$, let $V = \{ (g(X_1),\ldots,g(X_n)) : g \in \G \}$ denote the set of vectors corresponding to the values of functions in $\G$.
As above, for a fixed $\gamma$ and fixed minimal $\gamma$-covering subset $\mathcal{N}_{\gamma} \subseteq V$, for each $v \in V$, $p(v)$ will denote the closest vector to $v$ in $\mathcal{N}_{\gamma}$.  We will denote by $\E_{\xi}$ the conditional expectation over the $\xi_i$ variables, given $X_1,\ldots,X_n$.  Note that
\begin{align*}
&\frac{1}{n}\E_{\xi}\max\limits_{v \in V}\left(\sum\limits_{i = 1}^{n}\xi_{i}v_{i}- c|v_{i}|\right)
\\
&\le
\frac{1}{n}\E_{\xi}\max\limits_{v \in V}\left(\sum\limits_{i = 1}^{n}\xi_{i}\!\left(v_{i} \!-\! p(v)_{i}\right)\right) 
\\
&+\frac{1}{n}\E_{\xi}\max\limits_{v \in V}\left(\sum\limits_{i = 1}^{n}\frac{c}{4}|p(v)_{i}| \!-\! c|v_{i}|\right)
\\
&+
\frac{1}{n}\E_{\xi}\max\limits_{v \in V}\left(\sum\limits_{i = 1}^{n}\xi_{i}p(v)_{i} \!-\! \frac{c}{4}|p(v)_{i}|\right).
\end{align*}
The first term is $\lesssim \frac{\gamma}{n}$ by the $\gamma$-cover property and the fact that $|\xi_i| \lesssim 1$.
Furthermore, as in the proof of Lemma~\ref{expectmaxnew}, the second term is at most $\frac{c}{4} \frac{\gamma}{n}$.
Now we analyze the last term carefully.  First we use the standard peeling argument. Given a set $W$ of vectors we define $W[a, b] = \{w \in W|a \le \rho_{H}(w, 0) < b\}$.  
\begin{align*}
&\frac{1}{n}\E_{\xi}\max\limits_{v \in V}\left(\sum\limits_{i = 1}^{n}\xi_{i}p(v)_{i}- \frac{c}{4}|p(v)_{i}|\right) 
\\
&=\frac{1}{n}\E_{\xi}\max\limits_{v \in \mathcal{N}_{\gamma}}\left(\sum\limits_{i = 1}^{n}\xi_{i}v_{i}- \frac{c}{4}|v_{i}|\right) 
\\
&\le \frac{1}{n}\E_{\xi}\max\limits_{v \in \mathcal{N}_{\gamma}[0, 2\gamma/c]}\left(\sum\limits_{i = 1}^{n}\xi_{i}v_{i}- \frac{c}{4}|v_{i}|\right) 
\\
&+ \frac{1}{n}\sum\limits_{k = 1}^{\infty}\E_{\xi}\max\limits_{\mathcal{N}_{\gamma}\left[2^{k}\gamma/c, 2^{k + 1}\gamma/c\right]}\left(\sum\limits_{i = 1}^{n}\xi_{i}v_{i}- \frac{c}{4}|v_{i}|\right)_{+}
\end{align*}
The first term is upper bounded by $\frac{2\log(\mathcal{M}^{\text{loc}}_{1}(\F, \gamma, n,c))}{cn}$ by Lemma~\ref{expectmax} and by noting that 
$|\mathcal{N}_{\gamma}[0, 2\gamma/c]| \leq \mathcal{M}_{1}(\mathcal{B}_{H}(0, (2\gamma)/c,\{X_1,\ldots,X_n\}), (2\gamma)/2) \leq \mathcal{M}^{\text{loc}}_{1}(\F,\gamma,n,c)$.
Now we upper-bound the second term. We start with an arbitrary summand. 
For $\lambda = \frac{c}{8}$, we have

\begin{align*}
&\E_{\xi}\max\limits_{v \in \{0\} \cup \mathcal{N}_{\gamma}\left[2^{k}\gamma/c, 2^{k + 1}\gamma/c\right]}\left(\sum\limits_{i = 1}^{n}\xi_{i}v_{i}- \frac{c}{4}|v_{i}|\right)
\\
&\le \frac{1}{\lambda}\ln\E_{\xi}\max\limits_{v \in \{0\} \cup\mathcal{N}_{\gamma}\left[2^{k}\gamma/c, 2^{k + 1}\gamma/c\right]}\exp\left\{\sum\limits_{i = 1}^{n}\lambda\xi_{i}v_{i} - \frac{\lambda c}{4} |v_{i}|\right\}
\\
&\le
\frac{1}{\lambda}\ln\left(\sum\limits_{v \in \mathcal{N}_{\gamma}\left[2^{k}\gamma/c, 2^{k + 1}\gamma/c\right]}\E_{\xi}\exp\left\{\sum\limits_{i = 1}^{n}\lambda\xi_{i}v_{i} - \frac{\lambda c}{4} |v_{i}|\right\}  + 1\right)
\\
&\le
\frac{1}{\lambda}\ln\left(\left|\mathcal{N}_{\gamma}\left[2^{k}\gamma/c, 2^{k + 1}\gamma/c\right]\right|\exp\left\{2^{k-2}\gamma(4\lambda^2 - \lambda c)/c\right\}  + 1\right)
\\
&\le
\frac{1}{\lambda}\ln\left(\left|\mathcal{N}_{\gamma}\left[0, 2^{k + 1}\gamma/c\right]\right|\exp\left\{2^{k-2}\gamma(4\lambda^2 - \lambda c)/c\right\}  + 1\right)
\\
&\le
\frac{1}{\lambda}\ln\left(\left(\mathcal{M}^{\text{loc}}_{1}(\F, 2\gamma, n, c)\right)^{2^{k + 1}}\exp\left\{2^{k-2}\gamma(4\lambda^2 - \lambda c)/c\right\}  + 1\right).
\end{align*}
Here we used that any minimal covering is also a packing, and $\left|\mathcal{M}_{\gamma}\left[0, 2^{k + 1}\gamma/c\right]\right| \le \left|\mathcal{M}^{\text{loc}}_{1}(\F, 2\gamma, n, c)\right|^{2^{k + 1}}$, where $\mathcal{M}_{\gamma}$ is a $\gamma$-packing (by the arguments as in Lemma $2.2$ in \cite{Mendelson15} and monotonicity of the local entropy \eqref{localentropy} with respect to the radius).
We fix $\gamma = K\gamma^{\text{loc}}_{c, c}(n)$ for some $K > 2$. Observe that local entropy is nonincreasing and $K\gamma^{\text{loc}}_{c, c}(n) > 2\gamma^{\text{loc}}_{c,c}(n) \geq \gamma^{\text{loc}}_{c, c}(n) + 1$. Thus,
\begin{align*}
&\frac{1}{\lambda}\ln\left(\exp\left(2^{k + 1}\log\left(\mathcal{M}^{\text{loc}}_{1}(V, 2K\gamma^{\text{loc}}_{c, c}(n), n, c)\right)\! +\! 2^{k-2}K\gamma^{\text{loc}}_{c, c}(n)(4\lambda^2 - \lambda c)/c\right)\! +\! 1\right)
\\
&\le\frac{1}{\lambda}\ln\left(\exp\left(2^{k + 1}c(\gamma^{\text{loc}}_{c, c}(n) + 1) + 2^{k-2}K\gamma^{\text{loc}}_{c, c}(n)(4\lambda^2 - \lambda c)/c\right)  + 1\right).
\end{align*}
Then we have
\begin{align*}
&\sum\limits_{k = 1}^{\infty}\frac{8}{c}\ln\left(\exp\left(2^{k + 1}\log\left(\mathcal{M}^{\text{loc}}_{1}(\G, 2K\gamma^{\text{loc}}_{c, c}(n), n, c)\right)\right)\exp\left(-2^{k - 6}Kc\gamma^{\text{loc}}_{c, c}(n)\right)+ 1\right)
\\
&\le \sum\limits_{k = 1}^{\infty}\frac{8}{c}\ln\left(\exp\left(2^{k + 2}c\gamma^{\text{loc}}_{c, c}(n) -2^{k - 6}Kc\gamma^{\text{loc}}_{c, c}(n)\right)+ 1\right).
\end{align*}
We set $K = 2^{9}$ and have 
$\sum\limits_{k = 1}^{\infty}\ln\left(\exp\left(2^{k + 2}c\gamma^{\text{loc}}_{c, c}(n) -2^{k - 6}Kc\gamma^{\text{loc}}_{c, c}(n)\right)+ 1\right) \le C,
$
where $C > 0$ is an absolute constant. Here we used that $\ln(x + 1) \le x$ for $x > 0$ and $c\gamma^{\text{loc}}_{c, c} \gtrsim 1$.
Finally, we have
\begin{equation*}
\frac{1}{n}\E_{\xi}\max\limits_{v \in V}\left(\sum\limits_{i = 1}^{n}\xi_{i}v_{i}- c|v_{i}|\right)
\!\lesssim\!
\frac{\gamma^{\text{loc}}_{c, c}(n)}{n}
\!+\!
\frac{\log(\mathcal{M}^{\text{loc}}_{1}(\F, \gamma^{\text{loc}}_{c, c}(n), n, c))}{cn}
\!+\!
\frac{1}{cn} \!\lesssim\! \frac{\gamma^{\text{loc}}_{c, c}(n)}{n}.
\end{equation*}
\end{proof}

Now we present the proof of the deviation bound in Theorem~\ref{mainupperbound}.

\begin{proof}[Theorem \ref{mainupperbound} Deviation Bound]
We will provide a detailed outline of the proof. This proof technically repeats the arguments from our previous results. The constants will be denoted by $c_{i}$ for $i \in \mathbb{N}$. The idea is to combine the technique we previously used for Theorem \ref{mainupperbound} in expectation with the symmetrization lemma (Lemma \ref{symmetrizationdev}). Once again, let $\hat{f}$ be any ERM and $\hat{g}$ be a corresponding function in the excess loss class $\G_{\mathcal{Y}}$. We have $R(\hat{f}) - R(f^{*}) = P\hat{g}$ and $P_{n}\hat{g} \le 0$. Then for any $c > 0$
\[
R(\hat{f}) - R(f^{*}) \le P\hat{g} - (1 + c)P_{n}\hat{g} \le \sup\limits_{g \in \G_{\mathcal{Y}}}(Pg - (1+c)P_{n}g).
\]
Now due to the fact that $\G_{\mathcal{Y}}$ is a $(\frac{1}{h}, 1)$-Bernstein class we have, using Lemma \ref{symmetrizationdev},
\[
\Prob\left(\sup\limits_{g \in \G_{\mathcal{Y}}}(P - (1 + c_1)P_{n})g \ge t\right) \le 2\Prob\left(\sup\limits_{g \in \G_{\mathcal{Y}}}((1\! +\! c_{2})P'_{n}\! -\! (1\! +\! c_{1})P_{n})g\! \ge\! t/2\right),
\]
provided that  $0 < c_2 < c_1 $ and $t \ge \frac{1}{nh}\frac{(1 + c_{2})^{2}}{c_{2}}$. Now we use the same argument as in the proof of Proposition \ref{coveringbound}. Specifically, to control the deviation of the value $\sup\limits_{g \in \G_{\mathcal{Y}}}(P'_{n} - (1 + c_3)P_{n})g$ it is enough to control the deviation of
\begin{equation}
\label{devterm}
\sup\limits_{g \in \G_{\mathcal{Y}}}\left(\frac{1}{n}\sum\limits_{i = 1}^{n}\varepsilon_{i}g(X_{i}, Y_{i}) - c_{4}P_{n}g\right).
\end{equation}
Now we use the second claim of Lemma \ref{contraction}. To control \eqref{devterm} it is enough to upper bound 
\begin{align*}
&\Prob_{\varepsilon}\left(\sup\limits_{f'\in \F^*}\left(\sum\limits_{i = 1}^{n}\varepsilon_{i}f'(X_i) - \frac{1}{2}hc_4|f'(X_i)|\right) \ge \frac{x}{2}\right)
\\
&\quad+ \Prob_{\xi}\left(\sup\limits_{g'\in \G_{f^*}}\left(\sum\limits_{i = 1}^{n}\xi_{i}g'(X_i) - \frac{1}{3}hg'(X_i)\right) \ge \frac{x}{3c_4}\right).
\end{align*}
Both summands are analyzed similarly. We proceed with the second one. Fix $\gamma \in \mathbb{N}$ and use the decomposition as in the beginning of the proof of Lemma \ref{localization}. Now the problem is reduced to the analysis of a $\gamma$-covering as before:
\begin{align*}
&\sup\limits_{g' \in\mathcal{G}_{f^{*}}}\left(\sum\limits_{i = 1}^{n}\xi_{i}g'(X_{i}) - \frac{h}{3}g'(X_i)\right) 
\\
&\le \gamma\left(1 + \frac{h}{12}\right)
\!+\!
\sup\limits_{g' \in\mathcal{G}_{f^{*}}}\left(\sum\limits_{i = 1}^{n}\xi_{i}p(g'(X_{i})) - \frac{h}{12}p(g'(X_i))\right). 
\end{align*}
The first term is deterministic. The concentration of the last term is given by a combination of Chernoff bound (as in Proposition \ref{coveringbound}) and an upper bound for the exponential moment of 
\[
\sup\limits_{g' \in\mathcal{G}_{f^{*}}}\left(\sum\limits_{i = 1}^{n}\xi_{i}p(g'(X_{i})) - \frac{h}{12}p(g'(X_i))\right)
\]
from the proof of Lemma \ref{localization}.
\end{proof}
\end{document}